\documentclass[12pt]{amsart}

\usepackage{customstyle} % throw all the ugly in the .sty file

\title{Moduli Theory of the $r$-Braid Arrangement} % title here
\author{Vance Blankers} 
\address[V. Blankers]{Department of Mathematics\\
Northeastern University\\
Boston, MA 02115\\
USA}
\email{v.blankers@northeastern.edu}

\author{Emily Clader} 
\address[E. Clader]{Department of Mathematics\\
San Francisco State University\\
San Francisco, CA 94132\\
USA}
\email{eclader@sfsu.edu}

\author{Iva Halacheva} 
\address[I. Halacheva]{Department of Mathematics\\
Northeastern University\\
Boston, MA 02115\\
USA}
\email{i.halacheva@northeastern.edu}

\author{Haggai Liu}
\address[H. Liu]{Department of Mathematics\\
Simon Fraser University\\
Burnaby, BC V5A 1S6 \\
Canada}
\email{haggail@sfu.ca}

\author{Dustin Ross}
\address[D. Ross]{Department of Mathematics\\
San Francisco State University\\
San Francisco, CA 94132\\
USA}
\email{rossd@sfsu.edu}

\begin{document}

\allowdisplaybreaks % this stops the align environments from ruining spacing elsewhere

\begin{abstract}
We describe a family of hyperplane arrangements depending on a positive integer parameter $r$, which we refer to as the $r$-braid arrangements, and which can be viewed as a generalization of the classical braid arrangement. The wonderful compactification of the braid arrangement (with respect to its minimal building set) is well-known to yield the moduli space $\overline{\mathcal{M}}_{0,n}$, and, in this work, we generalize this result, constructing a moduli space $\overline{\mathcal{M}}^r_{n}$ of certain genus-zero curves with an order-$r$ involution that we identify with the corresponding wonderful compactification of the $r$-braid arrangement. The resulting space is a variant of the previously studied moduli space $\overline{\mathcal{L}}^r_n$ \cite{cdlr_permutohedral}, related via a change of weights on the markings.
\end{abstract}

\maketitle

\section{Introduction}\label{sec:intro}

For an arrangement $\mcA$ of hyperplanes in $\P^N$, a wonderful compactification is, roughly speaking, a way of replacing $\P^N$ by a different ambient variety in such a way that the complement of the hyperplanes is preserved but the arrangement $\mcA$ itself is replaced by a divisor with normal crossings. There are a number of such compactifications, given by blowing up $\P^N$ (in a prescribed order) along certain sets of subvarieties obtained via intersecting subsets of $\mcA$, known as building sets. The collection of building sets is partially ordered by inclusion, and it has a unique maximal and minimal element.

An important special case is when $\mcA$ is the braid arrangement, which consists of the hyperplanes
\[H_{ij} = \{[x_1: \cdots: x_n] \; | \; x_i = x_j\} \subseteq \P^{n-1}\]
for all $0 \leq i < j \leq n$, under the convention that $x_0 = 0$. In this case, the wonderful compactification with respect to the minimal building set coincides with the moduli space $\Mbar_{0,n+2}$ of genus-zero stable curves \cite[Section 4.3]{DCP}.

For any positive integer $r$, there is a natural generalization of the braid arrangement, which we refer to as the $r$-braid arrangement, consisting of the hyperplanes
\[H_{ij}^k = \{[x_1: \cdots: x_n] \; | \; x_i = \zeta^k x_j\} \subseteq \P^{n-1}\]
for all $1 \leq i \leq j \leq n$ and $0 \leq k \leq r-1$, where $\zeta$ denotes a fixed primitive $r$th root of unity. The primary goal of this paper is to give a modular interpretation of the wonderful compactification of the $r$-braid arrangement with respect to its minimal building set.

Specifically, we construct a moduli space $\Mbar^r_n$ parametrizing genus-zero stable curves with an order-$r$ automorphism $\sigma$, equipped with two marked fixed points of $\sigma$ and $n$ marked orbits of $\sigma$.  Our main theorem is the following.

\begin{theorem}
\label{thm:main}
Let $r \geq 2$ and $n \geq 1$. Then the wonderful compactification of the $r$-braid arrangement in $\P^{n-1}$, with respect to its minimal building set, is the moduli space $\Mbar^r_n$.
\end{theorem}

\begin{remark}
This space is closely related to the moduli space $\Lbar^r_n$ constructed and studied in \cite{cdlr_permutohedral, cdlr_wonderful, cdehl_multimatroids}, which has interesting connections to generalizations of the permutohedron and matroids. More precisely, the marked fixed points and all but one of the marked orbits in $\Lbar^r_n$ are ``light'' in the sense that they are allowed to coincide with one another, whereas the marked points in $\Mbar^r_n$ are required to be all distinct. In this way, $\Mbar^r_n$ can be viewed as relating to $\Lbar^r_n$ in much the same way that $\Mbar_{0,n}$ relates to the Losev--Manin moduli space $\Lbar_n$, which parametrizes genus-zero curves with two ``heavy'' (non-coinciding) marked points and $n$ light marked points.
\end{remark}

The strategy for the proof of Theorem~\ref{thm:main} is as follows. First, we study the combinatorics of the intersection lattice of the $r$-braid arrangement in $\P^{n-1}$, relating it to a certain poset of decorated graphs. This allows us to describe the minimal building set explicitly as consisting of all subspaces spanned by a subset of the coordinate points
\[[1:0: \cdots: 0], \; [0:1:0: \cdots: 0], \; \ldots,\; [0: \cdots : 0:1]\]
and at most one of the $r^{n-1}$ points whose coordinates are all $r$th roots of unity. In particular, the desired wonderful compactification $\Mbar^r_n$ is the blow-up of $\P^{n-1}$ along all such subspaces, in any inclusion-increasing order.

From here, we leverage the known perspective on both the moduli space of curves and the Losev--Manin moduli space $\Lbar_n$ as similar iterated blow-ups.  A key fact about Losev--Manin space is that, by virtue of its stability condition, the genus-zero curves being parametrized are chains of projective lines, as opposed to the arbitrary trees of projective lines that are parametrized by $\Mbar_{0,n+2}$. This allows us to define an ``$r$th power'' map $q: \Lbar_n \rightarrow \Lbar_n$ given by $[x:y] \mapsto [x^r: y^r]$ on each $\P^1$ in the chain. Using $q$, together with the above explicit perspective on the wonderful compactification of the $r$-braid arrangement as an iterated blow-up, we identify the wonderful compactification with a fiber product
\[\xymatrix{
F \ar[r]\ar[d] & \Lbar_n\ar[d]^{q}\\
\Mbar_{0,n+2}\ar[r]^-{c} & \Lbar_n,}\]
in which $c$ is a weight-reduction morphism. This fiber diagram allows us to construct a universal family over $F$ of the objects parametrized by $\Mbar^r_n$ by combining the data of the universal families over $\Mbar_{0,n+2}$ and $\Lbar_n$. From here, the fact that $F$ is the requisite moduli space follows.

\subsection*{Acknowledgments}  The authors wish to thank the organizers of the 2024 workshop ``Combinatorics of Moduli of Curves,'' and the Banff International Research Station for hosting this workshop, which was where work on this project began.  E.C. was supported by NSF CAREER grant 2137060.  I.H. was supported by NSF DMS--2302664. D.R. was supported by NSF DMS--2302024 and DMS--2001439.

\section{Wonderful compactifications}\label{sec:wc}

Wonderful compactifications were introduced by De Concini and Procesi \cite{DCP} in the context of hyperplane arrangements in projective space and later generalized by Li \cite{LiLi} to more general arrangements of subvarieties in a smooth ambient variety. We review the necessary background in this section, referring the reader to such references as \cite{DCP, FY, LiLi} for further details.

\subsection{Basic combinatorics}

A {\bf hyperplane arrangement} in $\P^N$ is a finite collection
\[
\mcA = \{H_1, \ldots, H_k\}
\]
of hyperplanes $H_i \subseteq \P^N$. The {\bf intersection lattice} of $\mcA$ is the collection $\mcL_{\mcA}$ of all subsets
\[
H_I := \bigcap_{i \in I} H_i \subseteq \P^N
\]
for $I \subseteq [k] := \{1, \ldots, k\}$. We consider $\mcL_{\mcA}$ as a poset under \emph{reverse} inclusion, so that the unique minimal element is $\hat{0} = H_{\emptyset} = \P^N$.

By the {\bf complement} of an arrangement $\mcA$, we mean
\[
Y^{\circ} := \P^N \setminus \bigcup_{i=1}^k H_i.
\]
One compactification of $Y^{\circ}$ is, of course, $\P^N$ itself, and $\P^N \setminus Y^{\circ}$ is (the union of) the original arrangement. The hyperplanes in $\mcA$ may not form a normal crossings divisor, however, and a wonderful compactification is, roughly speaking, a compactification $\overline{Y}$ of $Y^{\circ}$ in which $\mcA$ is replaced by a divisor $\overline{Y} \setminus Y^{\circ}$ with normal crossings. The way to produce such a compactification is, unsurprisingly, by blowing up $\P^N$ at subspaces along which unwanted intersections of $\mcA$ occur.

Since some subsets of $\mcA$ may already intersect in normal crossings, one generally need only blow up at a subset of $\mcL_{\mcA}$ in order to obtain a compactification with the requisite ``wonderful'' properties (some of which we describe below). In particular, the subsets of $\mcL_{\mcA}$ that give rise to wonderful compactifications are as follows.

\begin{definition}
\label{def:buildingset}
Let $\mcA$ be a hyperplane arrangement in $\P^N$. A {\bf building set} of $\mcA$ is a subset $\mcG \subseteq \mcL_{\mcA} \setminus \{\hat{0}\}$ such that, for all $X \in \mcL_{\mcA} \setminus \{\hat{0}\}$, there is an isomorphism of posets
  \[[\hat{0}, X] \cong \prod_{\substack{\text{maximal }Y \in \mcG\\ \text{with }Y \leq X}} [\hat{0},Y].\]
Here, $\leq$ denotes the reverse-inclusion order in $\mcL_{\mcA}$, and $[A,B]$ denotes the interval
  \[[A,B] := \{C \in \mcL_{\mcA} \; | \; A \leq C \leq B\}.\]
\end{definition}

The condition of Definition~\ref{def:buildingset} is automatically satisfied when $X \in \mcG$, so $\mcG = \mcL_{\mcA} \setminus \{\hat{0}\}$ is a building set (the {\bf maximal building set}); this corresponds to the observation that one can obtain a wonderful compactification by blowing up $\P^N$ at {\it every} intersection of the hyperplanes in $\mcA$. On the other hand, if the arrangement includes distinct hyperplanes $H_1$ and $H_2$ for which $X = H_1 \cap H_2$ and no other hyperplane in the arrangement meets $X$, then it is straightforward to verify that
\[[\hat{0}, X] \cong [\hat{0}, H_1] \times [\hat{0}, H_2].\]
Therefore, a building set can include $H_1$ and $H_2$ without including $X = H_1 \cap H_2$; this corresponds to the observation that $H_1$ and $H_2$ already have normal crossings, so no blow-up along their intersection is needed. In Figure~\ref{fig:three-hyp-int-at-line}, we contrast this situation with the situation where there is a third hyperplane $H_3$ for which $H_1 \cap H_2 = H_1 \cap H_3 = H_2 \cap H_3$, in which case a building set must contain this (non-transverse) intersection $X$ since $[\hat{0}, X]$ does not decompose as a product. 

\begin{figure}
  \centering
  \begin{minipage}{0.4\linewidth}
  \includegraphics[width=0.8\linewidth]{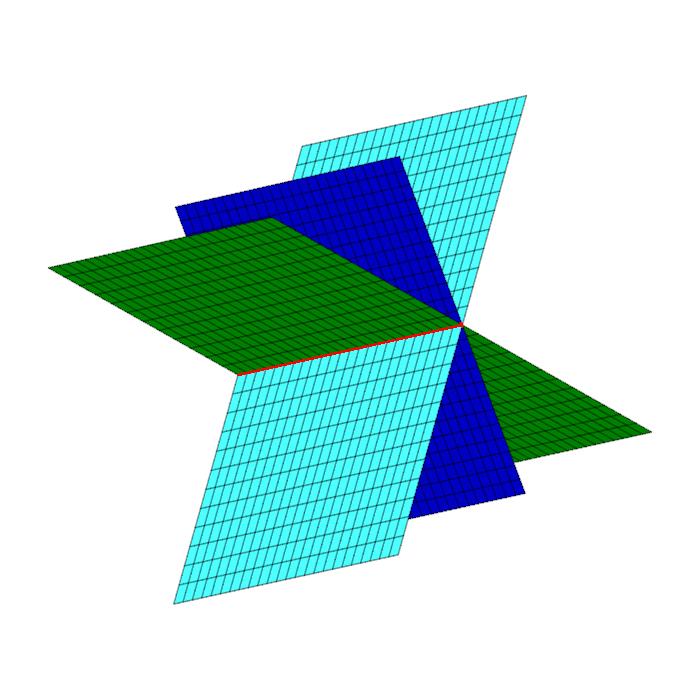}
  \end{minipage}
  \begin{minipage}{0.4\linewidth}
  \xymatrix{
    & X\ar@{-}[dl]\ar@{-}[d]\ar@{-}[dr] & \\
    H_1\ar@{-}[dr] & H_2\ar@{-}[d] & H_3\ar@{-}[dl] \\
    & \hat{0} & 
    }
  \end{minipage}
  \caption{Hyperplanes $H_1, H_2, H_3$ in three-space that intersect at a common line $X = H_1 \cap H_2 \cap H_3$, and the corresponding poset $[\hat{0}, X]$, which has rank two but does not decompose as a product of two rank-one posets.}
  \label{fig:three-hyp-int-at-line}
\end{figure}

Once a building set is chosen, the precise definition of a wonderful compactification is as follows.

\begin{definition}
Let $\mcA$ be a hyperplane arrangement in $\P^N$, and let $\mcG \subseteq \mcL_{\mcA} \setminus \{\hat{0}\}$ be a building set. Choose an ordering of the elements of $\mcG$ that is compatible with inclusion; that is, let
\[\mcG = \{G_1, \ldots, G_\ell\}\]
in which $i \leq j$ whenever $G_i \subseteq G_j$. Then the {\bf wonderful compactification} of $\mcA$ with respect to $\mcG$ is the variety $\overline{Y}_{\mcG}$ obtained by blowing up $\P^N$ along $G_1$, then blowing up the result along the strict transform of $G_2$, then blowing up the result along the strict transform of $G_3$, and so on.
\end{definition}

Among the ``wonderful'' properties of $\overline{Y}_{\mcG}$ are the following \cite[Section 3]{DCP}:
\begin{itemize}
\item Because $\overline{Y}_{\mcG}$ is obtained from $\P^N$ by blowing up only at subvarieties that do not meet $Y^{\circ}$, there is an inclusion
\[Y^{\circ} \hookrightarrow \overline{Y}_{\mcG}\]
as a dense open set; in other words, $\overline{Y}_{\mcG}$ is a compactification of $Y^{\circ}$.
\item The boundary $\overline{Y}_{\mcG} \setminus Y^{\circ}$ is a divisor with normal crossings.
\item The intersection-theoretic structure of the boundary can be read off combinatorially: namely, $\overline{Y}_{\mcG} \setminus Y^{\circ}$ is a union of divisors $D_G$ for nonempty $G \in \mcG$, and $D_{G_1} \cap \cdots \cap D_{G_k}$ is nonempty if and only if $\{G_1, \ldots, G_k\}$ forms a ``$\mcG$-nested set.''
\end{itemize}
The notion of ``$\mcG$-nested set'' is entirely combinatorial (see \cite[Definition 2.4]{DCP} or \cite[Definition 3.2]{Feichtner}); for example, when $\mcG$ is the maximal building set, $\mcG$-nested sets are simply chains in the poset $\mcL_{\mcA} \setminus \{\hat{0}\}$.

\subsection{The minimal building set}

In this paper, we will be primarily concerned with the minimal building set. To describe this more explicitly, refer to $X \in \mcL_{\mcA} \setminus \{\hat{0}\}$ as {\bf decomposable} if there is an isomorphism of posets
\begin{align}\label{eqn:decomp-join}
\nonumber    [\hat{0}, Y_1] \times [\hat{0}, Y_2] &\overset{\cong}{\rightarrow} [\hat{0}, X]\\
    \ 
    (H_1, H_2) &\mapsto H_1\lor H_2 = H_1\cap H_2
\end{align}
for some $Y_1, Y_2 < X$, and otherwise, refer to $X$ as {\bf indecomposable}. 

Note that, for any $X \in \mcL_{\mcA} \setminus \{\hat{0}\}$, one has
\[
[\hat{0}, X] \cong \prod_{\substack{\text{maximal}\\ \text{indecomposable }Y \leq X}} [\hat{0}, Y].
\]
This fact follows from two observations:
\begin{enumerate}
\item Any interval $[\hat{0}, X]$ can successively be decomposed into products of subintervals until the upper bounds on those subintervals are indecomposable; and
\item Any element of a product of intervals that is not contained in one of the intervals is decomposable, as it can be decomposed by its factors in the product.
\end{enumerate}
Thus, it follows that the set of indecomposable elements forms a building set. Conversely, it is straightforward from Definition~\ref{def:buildingset} to see that any building set must contain all indecomposable elements, so the {\bf minimal building set} is the set
\[\mcG_{\text{min}} = \Big\{\text{indecomposable elements of } \mcL_{\mcA} \setminus \{\hat{0}\}\Big\}.\]

As an illustrative example, let us determine the minimal building set of the braid arrangement.  These results are well-known to experts, but we summarize them here as they provide a helpful template for the setting of the $r$-braid arrangement to come.

\begin{example}
\label{ex:braid}
Recall from the introduction that the braid arrangement in $\P^{n-1}$ is the arrangement $\mcA_n$ consisting of all hyperplanes
\[H_{ij} = \{[x_1: \cdots: x_n] \; | \; x_i = x_j\} \subseteq \P^{n-1}\]
for $0 \leq i < j \leq n$, where we set $x_0 = 0$. Elements of the intersection lattice $\mcL_{\mcA_n}$ can be naturally identified with set partitions of $\{0,1,\ldots, n\}$, in which the set partition
\[\{0,1, \ldots, n\} = S_1 \sqcup \cdots \sqcup S_\ell\]
corresponds to the linear subspace
\begin{equation}
  \label{eq:Gminbraid}
  \Big\{[x_1: \cdots: x_n] \; \big| \; x_i = x_j \text{ if } \{i,j\} \subseteq S_k \text{ for some }k\Big\}.
\end{equation}
Let us denote a set partition by $\mathbf{S} = \{S_1, \ldots, S_m\}$, where we omit any parts of size $1$ from the notation (since they do not contribute to the condition defining \eqref{eq:Gminbraid}), and let us denote the corresponding subspace by $X_{\mathbf{S}}$. If $m \geq 2$, then 
\[[\hat{0}, X_{\mathbf{S}}] \cong [\hat{0}, X_{\mathbf{S}_1}] \times [\hat{0}, X_{\mathbf{S}_2}],\]
in which $\mathbf{S}_1$ and $\mathbf{S}_2$ are given by removing $S_1$ and $S_2$, respectively, from $\mathbf{S}$.

Thus, the indecomposable elements of $\mcL_{\mcA_n}$ correspond to set partitions with only one non-singleton part $S$. As subsets of $\P^{n-1}$, these are of the form
\[
X_S:= X_{\{S\}} = \Big\{[x_1: \cdots: x_n] \; \big| \; x_i = x_j \text{ if } \{i,j\} \subseteq S\Big\}
\]
for any nonempty subset $S \subseteq \{0,1,\ldots, n\}$. Setting $p_0 = [1:1: \cdots: 1]$ and, for $1 \leq i \leq n$,
\[p_i = [0: \cdots: 0: 1: 0: \cdots: 0]\]
(with a $1$ in the $i$th spot), then $X_S$ can equivalently be expressed as the linear subspace spanned by $p_i$ with $i \notin S$.

One finds, then, that the wonderful compactification of $\mcA_n$ with respect to its minimal building set is the iterated blow-up of $\P^{n-1}$ along all subspaces spanned by subsets of $\{p_0, p_1, \ldots, p_n\}$, in inclusion-increasing order. This precisely coincides with Kapranov's description \cite{Kapranov} of $\Mbar_{0,n+2}$.
\end{example}

In the next section, we adapt the argument in the above example to give an explicit description of the minimal building set of the $r$-braid arrangement.

\section{The $r$-braid arrangement and its minimal building set}
\label{sec:minbuildingset}

Fix a primitive $r$th root of unity $\zeta$. The {\bf $r$-braid arrangement} in $\P^{n-1}$ is the arrangement $\mcA^r_n$ consisting of all hyperplanes of the form
\[H^k_{ij} := \{[x_1: \cdots : x_n] \; | \; x_i = \zeta^k x_j\} \subseteq \P^{n-1}\]
for $1 \leq i \leq j \leq n$ and $0 \leq k \leq r-1$. Note, in particular, that we allow $i=j$, but in this case we must insist that $k \neq 0$ (as $H^0_{ii}$ is all of $\P^{n-1}$, not a hyperplane), and for such $k$,
\[H^k_{ii} =\{ [x_1: \cdots: x_n] \; | \; x_i = 0\}.\]

The goal of this section is to give an explicit description of the minimal building set of the $r$-braid arrangement along the lines of Example~\ref{ex:braid}. Here, as above, we denote by $p_1, \ldots, p_n$ the coordinate points in $\P^{n-1}$. Furthermore, we denote $\Z_r = \{0,1, \ldots, r-1\}$, and for any $\mathbf{a} = (a_1, \ldots, a_n) \in \Z_r^n$, we denote
\[p_{\mathbf{a}} = [\zeta^{a_1}: \zeta^{a_2}: \cdots: \zeta^{a_n}].\]
The description that we aim to prove is the following.

\begin{proposition}
\label{prop:minbuildingset}
The minimal building set of the $r$-braid arrangement in $\P^{n-1}$ consists of all subspaces spanned either by a subset $S$ of the coordinate points $p_1, \ldots, p_n$, or by $S \cup \{p_{\mathbf{a}}\}$ for some $\mathbf{a} \in \Z_r^n$.
\end{proposition}

To prove this, we first identify the intersection lattice $\mcL_{\mcA^r_n}$ with a purely combinatorial lattice, in much the same way that the intersection lattice of the braid arrangement is identified with the lattice of set partitions. The combinatorial objects in this case are somewhat richer; we refer to them as {\bf $(r,n)$-graphs} and introduce them below.

\subsection{The lattice of $(r,n)$-graphs}

Fix $r \geq 2$ and $n \geq 1$. In what follows, a {\bf complete graph} refers (as usual) to a simple graph with an edge between every pair of distinct vertices, and a {\bf complete $r$-flower} refers to a non-simple graph with $r$ edges between every pair of distinct vertices and $r-1$ self-edges at each vertex.

\begin{definition}
\label{def:rngraph}
An {\bf $(r,n)$-graph} is a graph $G$ with vertex set
\[V(G) = [n] = \{1, 2, \ldots, n\}\]
and a fixed edge labeling
\begin{align*}
   E(G) &\rightarrow \Z_r\\
   \{i,j\} &\mapsto k_{ij} = k_{ji}
\end{align*}
such that $G$ decomposes as a disjoint union of connected subgraphs
\[
G = G_0 \sqcup G_1 \sqcup \cdots \sqcup G_b
\]
for some $b \geq 0$, satisfying the following conditions:
\begin{enumerate}
\item $G_1, \ldots, G_b$ are nonempty complete graphs, and their edge labels are {\bf compatible} in the sense that for any triple of vertices $i_1 < i_2 < i_3$, we have
\[
k_{i_1i_3} \equiv (k_{i_1i_2} + k_{i_2i_3}) \mod r.
\]
\item $G_0$ is a (possibly empty) complete $r$-flower with the $r$ edges between distinct vertices labeled $0, 1, \ldots r-1$ and the $r-1$ self-edges at each vertex labeled $1, \ldots, r-1$.
\end{enumerate}

Let $\Gamma^r_n$ denote the poset of all $(r,n)$-graphs, where $G \leq H$ if $E(G) \subseteq E(H)$ and the edge labeling on $G$ agrees with that on $H$.  
\end{definition}

\begin{example}
\label{ex:rngraph}
An example of a $(3,7)$-graph is shown below, with vertices labeled in black and edges labeled in blue:

\tikzset{every picture/.style={line width=0.75pt}} %set default line width to 0.75pt        

\begin{center}
\begin{tikzpicture}[x=0.75pt,y=0.75pt,yscale=-1,xscale=1]

%Shape: Ellipse [id:dp8610133624326458] 
\draw  [color={rgb, 255:red, 0; green, 0; blue, 0 }  ,draw opacity=1 ][fill={rgb, 255:red, 0; green, 0; blue, 0 }  ,fill opacity=1 ] (90,119.78) .. controls (90,117.69) and (91.57,116) .. (93.5,116) .. controls (95.43,116) and (97,117.69) .. (97,119.78) .. controls (97,121.87) and (95.43,123.56) .. (93.5,123.56) .. controls (91.57,123.56) and (90,121.87) .. (90,119.78) -- cycle ;
%Curve Lines [id:da8343634202471739] 
\draw    (93.5,119.78) .. controls (69,123.8) and (34,104.8) .. (44,82.8) .. controls (54,60.8) and (93,83.8) .. (93.5,119.78) -- cycle ;
%Curve Lines [id:da8037641962017775] 
\draw    (93.5,123.56) .. controls (94,154.8) and (57,183.8) .. (43,160.8) .. controls (29,137.8) and (59,123.8) .. (93.5,119.78) ;

%Shape: Ellipse [id:dp42475663687806553] 
\draw  [color={rgb, 255:red, 0; green, 0; blue, 0 }  ,draw opacity=1 ][fill={rgb, 255:red, 0; green, 0; blue, 0 }  ,fill opacity=1 ] (151.48,119.78) .. controls (151.48,117.69) and (149.91,116) .. (147.98,116) .. controls (146.05,116) and (144.48,117.69) .. (144.48,119.78) .. controls (144.48,121.87) and (146.05,123.56) .. (147.98,123.56) .. controls (149.91,123.56) and (151.48,121.87) .. (151.48,119.78) -- cycle ;
%Curve Lines [id:da6197837619512122] 
\draw    (147.98,119.78) .. controls (172.48,123.8) and (207.48,104.8) .. (197.48,82.8) .. controls (187.48,60.8) and (148.48,83.8) .. (147.98,119.78) -- cycle ;
%Curve Lines [id:da47514397451152435] 
\draw    (147.98,123.56) .. controls (147.48,154.8) and (184.48,183.8) .. (198.48,160.8) .. controls (212.48,137.8) and (182.48,123.8) .. (147.98,119.78) ;

%Straight Lines [id:da28530112857197865] 
\draw    (93.5,119) -- (147.98,119) ;
%Curve Lines [id:da6904062250987233] 
\draw    (93.5,119) .. controls (106,86.8) and (132,85.8) .. (147.98,119) ;
%Curve Lines [id:da2850482805602663] 
\draw    (93.5,123.56) .. controls (109,153.8) and (134,156.8) .. (147.98,123.56) ;
%Shape: Polygon [id:ds15667561581492895] 
\draw   (301,80.8) -- (343,152.8) -- (261,151.8) -- cycle ;
%Shape: Circle [id:dp883052742378013] 
\draw  [fill={rgb, 255:red, 0; green, 0; blue, 0 }  ,fill opacity=1 ] (257.6,150.8) .. controls (257.6,148.93) and (259.12,147.4) .. (261,147.4) .. controls (262.88,147.4) and (264.4,148.93) .. (264.4,150.8) .. controls (264.4,152.68) and (262.88,154.21) .. (261,154.21) .. controls (259.12,154.21) and (257.6,152.68) .. (257.6,150.8) -- cycle ;
%Shape: Circle [id:dp6831355565041674] 
\draw  [fill={rgb, 255:red, 0; green, 0; blue, 0 }  ,fill opacity=1 ] (297.6,84.21) .. controls (297.6,82.33) and (299.12,80.8) .. (301,80.8) .. controls (302.88,80.8) and (304.4,82.33) .. (304.4,84.21) .. controls (304.4,86.09) and (302.88,87.61) .. (301,87.61) .. controls (299.12,87.61) and (297.6,86.09) .. (297.6,84.21) -- cycle ;
%Shape: Circle [id:dp6445935861747064] 
\draw  [fill={rgb, 255:red, 0; green, 0; blue, 0 }  ,fill opacity=1 ] (339.6,152.8) .. controls (339.6,150.93) and (341.12,149.4) .. (343,149.4) .. controls (344.88,149.4) and (346.4,150.93) .. (346.4,152.8) .. controls (346.4,154.68) and (344.88,156.21) .. (343,156.21) .. controls (341.12,156.21) and (339.6,154.68) .. (339.6,152.8) -- cycle ;

%Shape: Circle [id:dp41617260937936695] 
\draw  [fill={rgb, 255:red, 0; green, 0; blue, 0 }  ,fill opacity=1 ] (384.6,120.8) .. controls (384.6,118.93) and (386.12,117.4) .. (388,117.4) .. controls (389.88,117.4) and (391.4,118.93) .. (391.4,120.8) .. controls (391.4,122.68) and (389.88,124.21) .. (388,124.21) .. controls (386.12,124.21) and (384.6,122.68) .. (384.6,120.8) -- cycle ;
%Shape: Circle [id:dp45453854682594796] 
\draw  [fill={rgb, 255:red, 0; green, 0; blue, 0 }  ,fill opacity=1 ] (455.6,121.8) .. controls (455.6,119.93) and (457.12,118.4) .. (459,118.4) .. controls (460.88,118.4) and (462.4,119.93) .. (462.4,121.8) .. controls (462.4,123.68) and (460.88,125.21) .. (459,125.21) .. controls (457.12,125.21) and (455.6,123.68) .. (455.6,121.8) -- cycle ;
%Straight Lines [id:da6361732166457917] 
\draw    (388,120.8) -- (459,121.8) ;

% Text Node
\draw (80,121.4) node [anchor=north west][inner sep=0.75pt]    {$1$};
% Text Node
\draw (150.48,121.18) node [anchor=north west][inner sep=0.75pt]    {$3$};
% Text Node
\draw (244,148.4) node [anchor=north west][inner sep=0.75pt]    {$2$};
% Text Node
\draw (350,149.4) node [anchor=north west][inner sep=0.75pt]    {$4$};
% Text Node
\draw (296,59.4) node [anchor=north west][inner sep=0.75pt]    {$6$};
% Text Node
\draw (383,127.4) node [anchor=north west][inner sep=0.75pt]    {$5$};
% Text Node
\draw (453,129.4) node [anchor=north west][inner sep=0.75pt]    {$7$};
% Text Node
\draw (30,81.4) node [anchor=north west][inner sep=0.75pt]    {$\textcolor[rgb]{0.29,0.56,0.89}{1}$};
% Text Node
\draw (28,144.4) node [anchor=north west][inner sep=0.75pt]    {$\textcolor[rgb]{0.29,0.56,0.89}{2}$};
% Text Node
\draw (200.48,81.2) node [anchor=north west][inner sep=0.75pt]    {$\textcolor[rgb]{0.29,0.56,0.89}{1}$};
% Text Node
\draw (202,152.4) node [anchor=north west][inner sep=0.75pt]    {$\textcolor[rgb]{0.29,0.56,0.89}{2}$};
% Text Node
\draw (115,101.4) node [anchor=north west][inner sep=0.75pt]    {$\textcolor[rgb]{0.29,0.56,0.89}{1}$};
% Text Node
\draw (116,129.4) node [anchor=north west][inner sep=0.75pt]    {$\textcolor[rgb]{0.29,0.56,0.89}{2}$};
% Text Node
\draw (115,77.4) node [anchor=north west][inner sep=0.75pt]    {$\textcolor[rgb]{0.29,0.56,0.89}{{\displaystyle 0}}$};
% Text Node
\draw (325,106.4) node [anchor=north west][inner sep=0.75pt]    {$\textcolor[rgb]{0.29,0.56,0.89}{1}$};
% Text Node
\draw (267,103.4) node [anchor=north west][inner sep=0.75pt]    {$\textcolor[rgb]{0.29,0.56,0.89}{2}$};
% Text Node
\draw (293.48,156.2) node [anchor=north west][inner sep=0.75pt]    {$\textcolor[rgb]{0.29,0.56,0.89}{1}$};
% Text Node
\draw (417,102.4) node [anchor=north west][inner sep=0.75pt]    {$\textcolor[rgb]{0.29,0.56,0.89}{{\displaystyle 0}}$};
\end{tikzpicture}
\end{center}

\end{example}

The key point of Definition~\ref{def:rngraph} is that there is an $(r,n)$-graph associated to each element $X \in \mcL_{\mcA^r_n}$; for instance, the graph in Example~\ref{ex:rngraph} is the $(3,7)$-graph associated to
\[X = \{[x_1: \ldots: x_7] \; | \; x_1 = x_3 = 0, \;\; x_2 = \zeta x_4 = \zeta^2 x_6, \; \; x_5 = x_7\},\]
which can be viewed as the intersection of $H_{ij}^k$ for all $i \leq j$ such that there exists an edge from $i$ to $j$ labeled by $k$.  More precisely, we have the following result.

\begin{lemma}
\label{lem:isowithgraphs}
There is an isomorphism of posets
\begin{align*}
\mcL_{\mcA^r_n} &\rightarrow \Gamma^r_n\\
X &\mapsto G_X,
\end{align*}
in which $G_X$ is the graph on vertex set $[n]$ with an edge between vertices $i\leq j$ labeled by $k \in \Z_r$ if and only if $X \subseteq H^k_{ij}$.
\end{lemma}
\begin{proof}
First, we must check that $G_X$ is indeed an $(r,n)$-graph. To do so, let
\[
Z_X = \{ i \in [n] \; | \; x_i = 0 \text{ for all } x \in X\}.
\]
Then $X \subseteq H_{ij}^k$ for all $i,j \in Z_X$ and all $k \in \Z_r$, so $G_X$ contains the complete $r$-flower on vertex set $Z_X$ with edge labeling as required by condition (ii) of Definition~\ref{def:rngraph}. From here, let $G'$ be any other connected component of $G_X$. Then $G'$ is a simple graph, because if there were multiple edges (meaning $X \subseteq H^k_{ij} \cap H^{k'}_{ij}$ for some $k \neq k'$) or self-edges (meaning $X \subseteq H^k_{ii}$ for some $k \neq 0$), then their vertices would be in $Z_X$. Furthermore, $G'$ is a complete graph, since for any vertices $i \leq j \in V(G')$, the fact that $G'$ is connected means that there exists a path in $G'$ from $i$ to $j$. Denote the vertices in this path by
\[
i=i_1, i_2 \ldots, i_\ell = j.
\]
Then, by the definition of the edge labels on $G_X$, each $x \in X$ satisfies
\[
x_{i_m} = \zeta^{p_{i_mi_{m+1}}} x_{i_{m+1}}
\]
for $m=1, \ldots, \ell-1$, where
\[
p_{i_mi_{m+1}} := \begin{cases} k_{i_mi_{m+1}} & \text{ if } i_m<i_{m+1} \\ -k_{i_mi_{m+1}}  & \text{ if } i_m>i_{m+1}. \end{cases}
\]
Adding these equations for all $m$ shows that each $x \in X$ satisfies $x_i = \zeta^{p_{ij}} x_j$ with
\[
p_{ij} := \sum_{m=1}^{\ell-1} p_{i_mi_{m+1}}.
\]
This means that $X \subseteq H_{ij}^{p_{ij}}$, so $G'$ contains an edge from $i$ to $j$, proving that $G'$ is indeed a complete graph. Furthermore, applying this argument with $i = j$ shows that we have
\[
x_i = \zeta^{p_{ii}} x_i.
\]
Since $i \notin Z_X$, there exists some $x \in X$ with $x_i \neq 0$. Thus, the above equation is only possible if $p_{ii} \equiv 0 \mod r$, and applying this to paths of length $\ell =3$ proves that condition (i) of Definition~\ref{def:rngraph} is satisfied. This shows that $G_X$ is an $(r,n)$-graph, as claimed.

To see that the association $X \mapsto G_X$ is bijective, let $G$ be any $(r,n)$-graph, and let 
\[S_G = \{(i,j,k) \in [n] \times [n] \times \Z_r \; | \; i \leq j \text{ and } \exists \text{ an edge in } G \text{ from } i \text{ to } j \text{ labeled } k\}.\]
Setting
\[X_G = \bigcap_{(i,j,k) \in S_G} H_{ij}^k,\]
it is straightforward to see that the associations $X \mapsto G_X$ and $G \mapsto X_G$ are inverse.

Finally, to see that this bijection matches up the poset structures in $\mcL_{\mcA^r_n}$ and $\Gamma^r_n$, note that by definition, $X \leq Y$ in $\mcL_{\mcA^r_n}$ if and only if $X \supseteq Y$. Thus, if $G_X$ has an edge from $i$ to $j$ labeled $k$, then $X \subseteq H_{ij}^k$ and hence $Y \subseteq H_{ij}^k$. This shows that $G_Y$ also has an edge from $i$ to $j$ labeled $k$, so $G_X \leq G_Y$ in $\Gamma^r_n$. 
\end{proof}

In light of Lemma~\ref{lem:isowithgraphs}, we can determine the minimal building set of the $r$-braid arrangement by instead determining the indecomposable elements of the poset $\Gamma^r_n$. These are characterized by the next lemma. To state the lemma, we refer to a connected component of an $(r,n)$-graph as {\bf trivial} if it has no edges; in particular, note that if $G_0 \neq \emptyset$, then it is always nontrivial. The minimal element $\hat{0}$ of $\Gamma^r_n$ is the unique $(r,n)$-graph with no edges, meaning all of its connected components are trivial. 

\begin{lemma}
\label{lem:indecomposablegraphs}
An element $G \in \Gamma^r_n \setminus \{\hat{0}\}$ is indecomposable if and only if $G$ has exactly one nontrivial connected component.
\end{lemma}
\begin{proof}
We prove both directions by contrapositive. First, suppose that $G \in \Gamma^r_n \setminus \{\hat{0}\}$ has multiple nontrivial connected components $G_{i_1}, \ldots, G_{i_k}$. These are not themselves $(r,n)$-graphs since their vertex set is a proper subset of $[n]$, but for any such graph $H$, we denote by $\widetilde{H}$ the $(r,n)$-graph with edges and edge labels as in $H$, so that any vertices in $[n] \setminus V(H)$ constitute trivial connected components in $\widetilde{H}$. We claim that
\begin{equation}
\label{eq:graphiso}[\hat{0}, G] \cong \prod_{j=1}^k [\hat{0}, \widetilde{G_{i_j}}],
\end{equation}
via the isomorphism
\[H \mapsto \big( \widetilde{H \cap G_{i_j}} \big)_{j=1}^k\]
with inverse given by
\[(H_1, \ldots, H_k) \mapsto H_1 \vee \cdots \vee H_k\]
for $H_j \in [\hat{0}, \widetilde{G_{i_j}}]$. The fact that these are inverses is straightforward to check: given an element $H \in [\hat{0}, G]$ and denoting $H_j:= \widetilde{H \cap G_{i_j}}$, we have
\[E(H) = E(H_1) \cup \cdots \cup E(H_k)\]
(with matching edge labels) and hence $H = H_1 \vee \cdots \vee H_k$, whereas given an element $(H_1, \ldots, H_k) \in \prod [\hat{0}, \widetilde{G_{i_j}}]$ with $H_1 \vee \cdots \vee H_k = H$, we have
\[\widetilde{H \cap G_{i_j}} = H_j.\]
This proves \eqref{eq:graphiso} and hence verifies that $G$ is decomposable.

Conversely, suppose that $G$ is decomposable, so that 
\begin{equation}
\label{eq:graphiso2}
[\hat{0}, G] \cong [\hat{0}, K_1] \times [\hat{0}, K_2]
\end{equation}
for (possibly disconnected) subgraphs $\hat{0} < K_1, K_2 < G$. 
%{\color{red} I think that part of the definition of ``decomposable'' should be that this isomorphism is given by $H \mapsto (H \cap K_1, H \cap K_2)$ with inverse $(H_1, H_2) \mapsto H_1 \vee H_2$. If so, this should be stated in the definition.} 
The fact that $K_1, K_2 \neq \hat{0}$ means that they contain nontrivial connected components $H_1, H_2$, respectively, and we claim that these are contained in disjoint connected components $G_{i_1}, G_{i_2}$ of $G$. In other words, the claim is that there are no edges of $G$ connecting a vertex in $H_1$ to a vertex in $H_2$, and this is straightforward to see: if $e$ were such an edge, then the subgraph consisting only of the edge $e$ would be in $[\hat{0}, G]$ but not the join of a graph in $[\hat{0}, K_1]$ and a graph in $[\hat{0}, K_2]$, which would violate \eqref{eq:graphiso2}.
\end{proof}

\subsection{Proof of minimal building set}

Equipped with these results on $(r,n)$-graphs, we are prepared to prove the description of the minimal building set of the $r$-braid arrangement given above.

\begin{proof}[Proof of Proposition~\ref{prop:minbuildingset}]
By Lemmas~\ref{lem:isowithgraphs} and \ref{lem:indecomposablegraphs}, the indecomposable elements of the lattice $\mcL_{\mcA^r_n}$ are $X \in \mcL_{\mcA^r_n}$ such that $G_X$ has exactly one nontrivial connected component. These are precisely those $X$ of the form
\begin{equation}
\label{eq:case1}
X = \{[x_1: \cdots: x_n] \in \P^{n-1} \; | \; x_{i_1} = \cdots = x_{i_\ell} = 0\}
\end{equation}
for some $i_1, \ldots, i_\ell \in [n]$ (if the nontrivial connected component is the complete $r$-flower) or
\begin{equation}
\label{eq:case2}
X = \{[x_1: \cdots: x_n] \in \P^{n-1} \; | \; \zeta^{k_1}x_{i_1} = \cdots = \zeta^{k_\ell} x_{i_\ell}\}
\end{equation}
for some distinct $i_1, \ldots, i_\ell \in [n]$ and some $k_1, \ldots, k_\ell \in \Z_r$ (if the nontrivial connected component is not the complete $r$-flower). In the first case \eqref{eq:case1} above, $X$ is equal to the subspace spanned by the coordinate points $p_i$ for $i \notin \{i_1, \ldots, i_\ell\}$, and in the second case \eqref{eq:case2}, $X$ is equal to the subspace spanned by these same coordinate points together with the point $p_{\mathbf{a}}$ given by 
\[a_i = \begin{cases} -k_i \mod r & \text{ if } i \in \{i_1, \ldots, i_\ell\}\\
0 & \text{ if } i \notin \{i_1, \ldots, i_\ell\}.\end{cases}\]
Thus, the indecomposable elements of $\mcL_{\mcA^r_n}$ are those subspaces spanned by a subset of the coordinate points $\{p_1, \ldots, p_n\}$ and at most one root-of-unity point $p_{\mathbf{a}} = [\zeta^{a_1}: \cdots : \zeta^{a_n}]$, proving the proposition.
\end{proof}

Proposition~\ref{prop:minbuildingset} implies that the wonderful compactification of $\mcA^r_n$ with respect to its minimal building set is the iterated blow-up of $\P^{n-1}$ along all subspaces spanned by subsets of $\{p_1, \ldots, p_n\}$ and at most one element of $\{p_{\mathbf{a}} \; | \; \mathbf{a} \in \Z_r^n\}$, in any inclusion-increasing order. Toward the ultimate goal of giving this blow-up a modular interpretation, we first devote the next section to reframing it as a fiber product of the usual moduli space of curves $\Mbar_{0,n+2}$ with the Losev--Manin space $\Lbar_n$.

\section{The wonderful compactification as a fiber product}

We begin this section by recalling some preliminaries on Hassett's moduli spaces of curves with weighted marked points, following \cite{hassett_weighted_curves}.

\subsection{Background on Hassett spaces}
\label{subsec:Hassett}

For any weight vector $\w = (w_1, \ldots, w_n) \in \Q \cap (0,1]^n$ such that $\sum_{i=1}^n w_i > 2$, the moduli space $\Mbar_{0, \w}$ is a smooth Deligne--Mumford stack that provides an alternative compactification of $\M_{0,n}$ in which the marked points are allowed to coincide in a controlled way prescribed by $\w$. Specifically, elements of $\Mbar_{0,\w}$ are tuples $(C; p_1, \ldots, p_n)$ up to automorphism, where $C$ is a genus-zero curve with at worst nodal singularities, and $p_1, \ldots, p_n \in C$ are marked points with assigned weights $w_1,\ldots, w_n$, satisfying the following conditions:
\begin{enumerate}
\item the sum of the weights of any collection of coinciding marked points is $ \leq 1$;
\item the sum of the weights at all special points (i.e. marked points or half-nodes) of any irreducible component is $>2$, where half-nodes are given weight $1$.
\end{enumerate}
We refer to such curves as {\bf $\w$-stable}. Note, in particular, that $\Mbar_{0,n}$ is recovered as the special case where $w_i = 1$ for all $i$.

Part of what makes Hassett spaces interesting is that certain choices of weight vector give rise to toric moduli spaces. A simple special case is the weight vector
\[\w_\P = (\underbrace{\epsilon, \ldots, \epsilon}_{n+1 \text{ times}}, 1),\]
where $\epsilon = 1/n$. In this case, the stability condition (ii) ensures that a $\w_\P$-stable curve has a single $\P^1$ component. Choosing coordinates on this $\P^1$ in which the last marked point is $\infty$ and the second-to-last marked point is $0$, the data of a $\w_\P$-stable curve is a choice of the remaining $n$ marked points (up to automorphism via scaling), which are complex numbers that do not all coincide with $0$ by the choice of $\epsilon$. That is:
\[\Mbar_{0,\w_\P} = \P^{n-1}.\]

A more interesting toric case is given by the weight vector
\[
\w_{\text{LM}} := (\underbrace{\epsilon, \ldots, \epsilon}_{n \text{ times}},1,1),
\]
where again $\epsilon = 1/n$. The resulting moduli space is denoted
\[
\Lbar_n:= \Mbar_{0,\w_{\text{LM}}}
\]
and is referred to as {\bf Losev--Manin space} \cite{losev_manin}.  The stability condition (ii) now forces that $\w_{\text{LM}}$-stable curves are chains of projective lines with the last two marked points on the two ends:

\tikzset{every picture/.style={line width=0.75pt}} %set default line width to 0.75pt    
\[\begin{tikzpicture}[x=0.75pt,y=0.75pt,yscale=-0.8,xscale=0.8]
%Flowchart: Connector [id:dp6234739050650182] 
\draw  (43,135.75) .. controls (43,111.31) and (63.04,91.5) .. (87.75,91.5) .. controls (112.46,91.5) and (132.5,111.31) .. (132.5,135.75) .. controls (132.5,160.19) and (112.46,180) .. (87.75,180) .. controls (63.04,180) and (43,160.19) .. (43,135.75) -- cycle ;
%Curve Lines [id:da45780591358810163] 
\draw [dash pattern={on 4.5pt off 4.5pt}] (43.38,132.35) .. controls (57.22,122.09) and (70.7,119.94) .. (88.93,120.07) .. controls (107.15,120.2) and (126.57,128.26) .. (132.5,135.75) ;
%Curve Lines [id:da6243722985956413] 
\draw  (43.66,134.26) .. controls (46.29,137.61) and (58.34,141.33) .. (62.48,142.82) .. controls (66.62,144.31) and (85.6,147.79) .. (105.77,144.31) .. controls (125.94,140.83) and (127.97,139.11) .. (132.5,135.75) ;

%Flowchart: Connector [id:dp9734397246170519] 
\draw  (133,135.75) .. controls (133,111.31) and (153.04,91.5) .. (177.75,91.5) .. controls (202.46,91.5) and (222.5,111.31) .. (222.5,135.75) .. controls (222.5,160.19) and (202.46,180) .. (177.75,180) .. controls (153.04,180) and (133,160.19) .. (133,135.75) -- cycle ;
%Curve Lines [id:da4734070118799343] 
\draw [dash pattern={on 4.5pt off 4.5pt}] (133.38,132.35) .. controls (147.22,122.09) and (160.7,119.94) .. (178.93,120.07) .. controls (197.15,120.2) and (216.57,128.26) .. (222.5,135.75) ;
%Curve Lines [id:da24662515480480574] 
\draw  (133.66,134.26) .. controls (136.29,137.61) and (148.34,141.33) .. (152.48,142.82) .. controls (156.62,144.31) and (175.6,147.79) .. (195.77,144.31) .. controls (215.94,140.83) and (217.97,139.11) .. (222.5,135.75) ;

%Flowchart: Connector [id:dp9665382433185765] 
\draw  (223,135.75) .. controls (223,111.31) and (243.04,91.5) .. (267.75,91.5) .. controls (292.46,91.5) and (312.5,111.31) .. (312.5,135.75) .. controls (312.5,160.19) and (292.46,180) .. (267.75,180) .. controls (243.04,180) and (223,160.19) .. (223,135.75) -- cycle ;
%Curve Lines [id:da2667622685044222] 
\draw [dash pattern={on 4.5pt off 4.5pt}] (223.38,132.35) .. controls (237.22,122.09) and (250.7,119.94) .. (268.93,120.07) .. controls (287.15,120.2) and (306.57,128.26) .. (312.5,135.75) ;
%Curve Lines [id:da023500526851394588] 
\draw  (223.66,134.26) .. controls (226.29,137.61) and (238.34,141.33) .. (242.48,142.82) .. controls (246.62,144.31) and (265.6,147.79) .. (285.77,144.31) .. controls (305.94,140.83) and (307.97,139.11) .. (312.5,135.75) ;

%Flowchart: Connector [id:dp5424595351198411] 
\draw  (312,135.75) .. controls (312,111.31) and (332.04,91.5) .. (356.75,91.5) .. controls (381.46,91.5) and (401.5,111.31) .. (401.5,135.75) .. controls (401.5,160.19) and (381.46,180) .. (356.75,180) .. controls (332.04,180) and (312,160.19) .. (312,135.75) -- cycle ;
%Curve Lines [id:da13597853842837182] 
\draw [dash pattern={on 4.5pt off 4.5pt}] (312.38,132.35) .. controls (326.22,122.09) and (339.7,119.94) .. (357.93,120.07) .. controls (376.15,120.2) and (395.57,128.26) .. (401.5,135.75) ;
%Curve Lines [id:da44213722331629546] 
\draw  (312.66,134.26) .. controls (315.29,137.61) and (327.34,141.33) .. (331.48,142.82) .. controls (335.62,144.31) and (354.6,147.79) .. (374.77,144.31) .. controls (394.94,140.83) and (396.97,139.11) .. (401.5,135.75) ;

%Shape: Ellipse [id:dp02051357566439016] 
\draw [color={rgb, 255:red, 9; green, 114; blue, 237 } ,draw opacity=1 ][fill={rgb, 255:red, 19; green, 115; blue, 229 } ,fill opacity=1 ] (38.16,137.26) .. controls (38.16,134.5) and (40.62,132.26) .. (43.66,132.26) .. controls (46.7,132.26) and (49.16,134.5) .. (49.16,137.26) .. controls (49.16,140.02) and (46.7,142.26) .. (43.66,142.26) .. controls (40.62,142.26) and (38.16,140.02) .. (38.16,137.26) -- cycle ;
%Shape: Ellipse [id:dp2377325743357961] 
\draw [color={rgb, 255:red, 9; green, 114; blue, 237 } ,draw opacity=1 ][fill={rgb, 255:red, 19; green, 115; blue, 229 } ,fill opacity=1 ] (81.16,158.26) .. controls (81.16,155.5) and (83.62,153.26) .. (86.66,153.26) .. controls (89.7,153.26) and (92.16,155.5) .. (92.16,158.26) .. controls (92.16,161.02) and (89.7,163.26) .. (86.66,163.26) .. controls (83.62,163.26) and (81.16,161.02) .. (81.16,158.26) -- cycle ;
%Shape: Ellipse [id:dp21243274097313258] 
\draw [color={rgb, 255:red, 9; green, 114; blue, 237 } ,draw opacity=1 ][fill={rgb, 255:red, 19; green, 115; blue, 229 } ,fill opacity=1 ] (101.16,108.26) .. controls (101.16,105.5) and (103.62,103.26) .. (106.66,103.26) .. controls (109.7,103.26) and (112.16,105.5) .. (112.16,108.26) .. controls (112.16,111.02) and (109.7,113.26) .. (106.66,113.26) .. controls (103.62,113.26) and (101.16,111.02) .. (101.16,108.26) -- cycle ;
%Shape: Ellipse [id:dp9224796009296896] 
\draw [color={rgb, 255:red, 9; green, 114; blue, 237 } ,draw opacity=1 ][fill={rgb, 255:red, 19; green, 115; blue, 229 } ,fill opacity=1 ] (158.16,156.26) .. controls (158.16,153.5) and (160.62,151.26) .. (163.66,151.26) .. controls (166.7,151.26) and (169.16,153.5) .. (169.16,156.26) .. controls (169.16,159.02) and (166.7,161.26) .. (163.66,161.26) .. controls (160.62,161.26) and (158.16,159.02) .. (158.16,156.26) -- cycle ;
%Shape: Ellipse [id:dp007957213102574912] 
\draw [color={rgb, 255:red, 9; green, 114; blue, 237 } ,draw opacity=1 ][fill={rgb, 255:red, 19; green, 115; blue, 229 } ,fill opacity=1 ] (260.16,103.26) .. controls (260.16,100.5) and (262.62,98.26) .. (265.66,98.26) .. controls (268.7,98.26) and (271.16,100.5) .. (271.16,103.26) .. controls (271.16,106.02) and (268.7,108.26) .. (265.66,108.26) .. controls (262.62,108.26) and (260.16,106.02) .. (260.16,103.26) -- cycle ;
%Shape: Ellipse [id:dp08307966986533932] 
\draw [color={rgb, 255:red, 9; green, 114; blue, 237 } ,draw opacity=1 ][fill={rgb, 255:red, 19; green, 115; blue, 229 } ,fill opacity=1 ] (277.16,109.26) .. controls (277.16,106.5) and (279.62,104.26) .. (282.66,104.26) .. controls (285.7,104.26) and (288.16,106.5) .. (288.16,109.26) .. controls (288.16,112.02) and (285.7,114.26) .. (282.66,114.26) .. controls (279.62,114.26) and (277.16,112.02) .. (277.16,109.26) -- cycle ;
%Shape: Ellipse [id:dp31154535156948593] 
\draw [color={rgb, 255:red, 9; green, 114; blue, 237 } ,draw opacity=1 ][fill={rgb, 255:red, 19; green, 115; blue, 229 } ,fill opacity=1 ] (336.16,160.26) .. controls (336.16,157.5) and (338.62,155.26) .. (341.66,155.26) .. controls (344.7,155.26) and (347.16,157.5) .. (347.16,160.26) .. controls (347.16,163.02) and (344.7,165.26) .. (341.66,165.26) .. controls (338.62,165.26) and (336.16,163.02) .. (336.16,160.26) -- cycle ;
%Shape: Ellipse [id:dp17390150078947575] 
\draw [color={rgb, 255:red, 9; green, 114; blue, 237 } ,draw opacity=1 ][fill={rgb, 255:red, 19; green, 115; blue, 229 } ,fill opacity=1 ] (396,135.75) .. controls (396,132.99) and (398.46,130.75) .. (401.5,130.75) .. controls (404.54,130.75) and (407,132.99) .. (407,135.75) .. controls (407,138.51) and (404.54,140.75) .. (401.5,140.75) .. controls (398.46,140.75) and (396,138.51) .. (396,135.75) -- cycle ;

% Text Node
\draw (-5,124.7) node [anchor=north west][inner sep=0.75pt]  {$p_{n+1}$};
% Text Node
\draw (408,124.7) node [anchor=north west][inner sep=0.75pt]  {$p_{n+2}$};
\end{tikzpicture}
\]
Each component of such a curve has two distinguished special points (the half-nodes or the ``heavy'' points $p_{n+1}, p_{n+2}$), and one can choose coordinates so that these distinguished special points are $0$ and $\infty$, with $0$ being the one closer to the marked point $p_{n+1}$ and $\infty$ being the one closer to the marked point $p_{n+2}$. In these coordinates, we can view $p_1, \ldots, p_n \in \C$, and there is an action of $(\C^*)^n$ on $\Lbar_n$ in which $(t_1, \ldots, t_n)$ acts by sending the marked point $p_i$ in each $\P^1$ component to the point $t_ip_i$ in the same component. In fact, this action endows $\Lbar_n$ with the structure of a smooth toric variety: it is the toric variety associated to the braid fan \cite[Section 2]{losev_manin}.

Another perspective on $\Lbar_n$ and why it is toric is given in terms of the {\bf weight-reduction maps} between Hassett spaces. Namely, for any weight vectors $\w$ and $\w'$ with $w_i \geq w_i'$ for all $i$, there is a birational map
\[
\Mbar_{0,\w} \rightarrow \Mbar_{0,\w'}
\]
given by iteratively contracting components of a $\w$-stable curve that fail to be $\w'$-stable. Applying this to the weight vectors $\w_{\text{LM}}$ and $\w_{\P}$ above, one obtains a map
\[
d: \Lbar_n \rightarrow \P^{n-1}.
\]
Analyzing this map via \cite[Proposition 4.5]{hassett_weighted_curves}, one sees that it is an iterated blow-up of $\P^{n-1}$ along all subspaces spanned by subsets of the coordinate points, in dimension-increasing order. Because such subspaces are all torus-invariant, the fact that $\Lbar_n$ is toric follows.

Along the same lines, there is a weight-reduction map
\[
c: \Mbar_{0,n+2} \rightarrow \Lbar_n,
\]
which can be viewed as an iterated blow-up of $\Lbar_n$ along the strict transforms of all subspaces of $\P^{n-1}$ spanned by $[1:1: \cdots: 1]$ together with some subset of the coordinate points. The composition of the two reduction maps $d \circ c: \Mbar_{0,n+2} \rightarrow \P^{n-1}$ recovers the Kapranov description of $\Mbar_{0,n+2}$ as an iterated blow-up discussed in Example~\ref{ex:braid}. (The crucial difference between this blow-up and the one discussed in the previous paragraph is that, since $[1:1: \cdots: 1]$ is not torus-invariant, the resulting blow-up is not toric. In this sense, $\Lbar_n$ can be viewed as ``the last toric blow-up on the way to $\Mbar_{0,n+2}$.'')

\subsection{The wonderful compactification as a fiber product}

In this section, we realize the wonderful compactification of the $r$-braid arrangement in $\P^{n-1}$ as a fiber product $\Mbar_{0,n+2}\times_{\Lbar_n}\Lbar_n$. The key insight that we require is that the $r$-braid arrangement is the preimage of the braid arrangement under the map $\P^{n-1}\rightarrow\P^{n-1}$ that takes the $r$th power of each coordinate. Since blow-ups commute with flat base change, we might then try to conclude that the wonderful compactification of the $r$-braid arrangement is the fiber product $\Mbar_{0,n+2}\times_{\P^{n-1}}\P^{n-1}$, where the maps are the iterated blow-up $d \circ c: \Mbar_{0,n+2} \rightarrow \P^{n-1}$ and the $r$th power map. However, this is not quite right: what goes wrong is that the coordinate hyperplanes happen to lie in the ramification locus of the $r$th power map, so instead of blowing up along these subspaces, the fiber product $\Mbar_{0,n+2}\times_{\P^{n-1}}\P^{n-1}$ blows up along nonreduced schemes supported on the coordinate subspaces. To remedy this, we first blow-up the coordinate spaces, obtaining $\Lbar_n$, and then we use the $r$th power map on $\Lbar_n$ instead.

More precisely, since $\Lbar_n$ is a smooth toric variety, there is a natural map $q:\Lbar_n\rightarrow \Lbar_n$ that extends the $r$th power map on the torus. In the coordinate description of elements of $\Lbar_n$ from the previous subsection, $q$ preserves the underlying curve $C$ and the last two marked points, and it acts on the first $n$ marked points by $[x:y] \mapsto [x^r: y^r]$. Observe that $q$ is finite locally free of degree $(n-1)^r$: in the affine toric chart of any torus-invariant point of $\Lbar_n$, the pullback $q^*$ can be realized as the inclusion of polynomial rings
\begin{align*}
    \C[x_1,\dots,x_{n-1}]&\rightarrow\C[x_1,\dots,x_{n-1}]\\
    f(x_1,\dots,x_{n-1})&\mapsto f(x_1^r,\dots,x_{n-1}^r),
\end{align*}
which exhibits the codomain as a free module of rank $(n-1)^r$ over the domain. Being finite locally free implies, in particular, that $q$ is finite and flat. Note that the ramification locus of $q$ is the toric boundary of $\Lbar_n$.

With this notation in place, the main result of this section is the following.

\begin{proposition}
\label{prop:WCfiberproduct}
The wonderful compactification of $\mcA^r_n$ with respect to its minimal building set is the fiber product $\Mbar^r_n$ in the following diagram:
\[  
\xymatrix{
  \Mbar^r_n \ar[r]\ar[d] & \Lbar_n\ar[d]^{q}\\
  \Mbar_{0,n+2}\ar[r]^-{c} & \Lbar_n.}
\]
\end{proposition}

\begin{proof}
We begin by setting up notation. Viewing $\Lbar_n$ as the iterated blow-up of $\P^{n-1}$ along the coordinate subspaces of $\P^{n-1}$, let $\mathcal{Z}$ be the collection of strict transforms of the linear subspaces in $\P^{n-1}$ spanned by $[1:\dots:1]$ together with a set of coordinate points, and let $\mathcal{W}$ be the collection of strict transforms of the linear subspaces in $\P^{n-1}$ spanned by one of the points $p_{\mathbf{a}} = [\zeta^{a_1}: \cdots: \zeta^{a_n}]$ together with a set of coordinate points. As discussed in the previous subsection, $c:\Mbar_{0,n+2}\rightarrow\Lbar_n$ is the iterated blow-up along the elements of $\mathcal{Z}$, in any inclusion-increasing order.

We aim to prove that $\Mbar_n^r\rightarrow\Lbar_n$ is the iterated blow-up along the elements of $\mathcal{W}$. Once we have accomplished this, we can combine it with the fact that $\Lbar_n\rightarrow\P^{n-1}$ is the iterated blow-up along the coordinate subspaces of $\P^{n-1}$ to see that $\Mbar_n^r\rightarrow\P^{n-1}$ is the iterated blow-up along the collection of linear subspaces spanned by a subset of the coordinate points and at most one of the points $p_{\mathbf{a}}$. By Proposition~\ref{prop:minbuildingset}, this is precisely the minimal building set of $\mcA^r_n$, from which we conclude that $\Mbar^r_n$ is the wonderful compactification with respect to this building set, as desired.

To prove that $\Mbar_n^r\rightarrow\Lbar_n$ is the iterated blow-up along the elements of $\mathcal{W}$, begin by decomposing the map $c$ as follows:
\[
\Mbar_{0,n+2}=Y_{n-2}\stackrel{c_{n-3}}{\longrightarrow} Y_{n-3}\stackrel{c_{n-4}}{\longrightarrow}\cdots\stackrel{c_{1}}{\longrightarrow} Y_1\stackrel{c_{0}}{\longrightarrow} Y_0 =\Lbar_n,
\]
where the map $c_k$ is the blow-up of $Y_k$ along the union of the strict transforms of the $k$-dimensional elements of $\mathcal{Z}$. Denote the intermediate compositions of these blow-up maps by $\hat c_k=c_0\circ\dots\circ c_{k-1}:Y_k\rightarrow\Lbar_n$, and let $\widetilde Z_k\subseteq Y_k$ denote the union of the strict transforms of the $k$-dimensional elements of $\mathcal{Z}$:
\[
\widetilde Z_k=\overline{\hat c_k^{-1}(Z_k\setminus Z_{k-1})},
\]
where $Z_k\subseteq\Lbar_n$ is the union of the $k$-dimensional elements of $\mathcal{Z}$. With this notation, we have $Y_{k+1}=\mathrm{Bl}_{\widetilde Z_k}(Y_k)$.

Recursively define $X_0=\Lbar_n$ and $X_k=Y_k\times_{Y_{k-1}}X_{k-1}$, so that every square in the following diagram is a pullback:
\[
\xymatrix{
  \Mbar_n^r=X_{n-2} \ar[r]\ar[d] & X_{n-3} \ar[r]\ar[d] & \cdots\ar[r] & X_{1} \ar[r]\ar[d] & X_0=\Lbar_n\ar[d]^{q}\\
  \Mbar_{0,n+2}=Y_{n-2}\ar[r] & Y_{n-3}\ar[r] & \cdots \ar[r] & Y_1\ar[r] & Y_0=\Lbar_n.}
\]
We prove inductively that the map $X_k\rightarrow \Lbar_n$ is the iterated blow-up along the elements of $\mathcal{W}$ of dimension less than $k$; the desired result then follows from the $k=n-2$ case.

To prove the induction step, consider the diagram
\[  
\xymatrix{
  X_{k+1} \ar[r]^{b_k}\ar[d]^{q_{k+1}} & X_k\ar[d]^{q_k}\\
  Y_{k+1} \ar[r]^{c_k} & Y_k.}
\]
Since $q$ is flat and since flatness is preserved by base change, $q_k$ is flat. Furthermore, since blow-ups commute with flat base change (see \cite[Exercise 24.1.P]{vakil_rising}) and since $Y_{k+1}=\mathrm{Bl}_{\widetilde Z_k}(Y_k)$, it follows that $b_k$ is a blow-up:
\[
X_{k+1}=\mathrm{Bl}_{\widetilde Z_k\times_{Y_k} X_k}(X_k).
\]
It remains to study the blow-up locus $\widetilde Z_k\times_{Y_k} X_k\subseteq X_k$; we must argue that it is equal to the strict transform of the union of the $k$-dimensional elements of $\mathcal{W}$. More precisely, after unraveling definitions, we must argue that
\begin{equation}\label{eq:fiberequation}
\overline{\hat c_k^{-1}(Z_k\setminus Z_{k-1})}\times_{Y_k}X_k=\overline{\hat b_k^{-1}(W_k\setminus W_{k-1})},
\end{equation}
where $W_k$ is the union of the $k$-dimensional elements of $\mathcal{W}$ and $\hat b_k=b_0\circ \cdots \circ b_{k-1}$. By definition of $\mathcal{Z}$ and $\mathcal{W}$, it is not hard to see that $W_k=q^{-1}(Z_k)$, and it then follows that $W_k\setminus W_{k-1}=q^{-1}(Z_k\setminus Z_{k-1})$. Since $W_k\setminus W_{k-1}$ is disjoint from the blow-up locus of $\hat b_k$ and $Z_k\setminus Z_{k-1}$ is disjoint from the blow-up locus of $\hat c_k$, it then follows that
\[
\hat b_k^{-1}(W_k\setminus W_{k-1}) = q_{k+1}^{-1}(\hat c_k^{-1}(Z_k\setminus Z_{k-1})).
\]
Lastly, since $q$ is finite locally free, which is preserved by base change, so is $q_{k+1}$, and since none of the components of $Z_k$ are contained entirely within the ramification locus of $q_{k+1}$, the equality \eqref{eq:fiberequation} is a special case of the lemma below.
\end{proof}

\begin{lemma}\label{lemma:flf}
Let $q:X\rightarrow Y$ be a finite locally free map of irreducible varieties and let $Z\subseteq Y$ be a subset such that none of the irreducible components of the Zariski closure $\overline Z$ lie entirely within the ramification locus of $q$. Then
\[
\overline Z\times_Y X=\overline{q^{-1}(Z)}.
\]
\end{lemma}

\begin{remark}
    The lemma fails if a component of $\overline Z$ lies within the ramification locus of $q$. For example, if $X=Y=\A^1$, $q(x)=x^2$, and $Z=\{0\}$, then the fiber product is a point of multiplicity two supported at the origin, while the closure of the preimage is a point of multiplicity one supported at the origin.
\end{remark}

\begin{proof}[Proof of Lemma~\ref{lemma:flf}]
It suffices to prove the lemma in the case where $X$ and $Y$ are affine and $q^*:\C[Y]\rightarrow\C[X]$ is an inclusion that endows $\C[X]$ with the structure of a free module over $\C[Y]$, since our hypotheses imply that $Y$ has an affine open cover over which the restriction of $q$ has this form.  In this case,
\[
\overline Z\times_Y X=\mathrm{Spec}\left(\frac{\C[Y]}{\mathcal{I}(Z)}\otimes_{\C[Y]}\C[X]\right)=\mathrm{Spec}\left(\frac{\C[X]}{\mathcal{I}(Z)\C[X]}\right)
\]
and
\[
\overline{q^{-1}(Z)}=\Spec\left(\frac{\C[X]}{\mathcal{I}(q^{-1}(Z))}\right).
\]
Thus, we must prove that $\mathcal{I}(Z)\C[X]=\mathcal{I}(q^{-1}(Z))$. The inclusion $\mathcal{I}(Z)\C[X]\subseteq\mathcal{I}(q^{-1}(Z))$ simply follows from the fact that the inclusion $q^*: \C[Y] \rightarrow \C[X]$ maps $\mathcal{I}(Z)$ into $\mathcal{I}(q^{-1}(Z))$, so the crux of the lemma is proving the other inclusion.

Let $g\in \mathcal{I}(q^{-1}(Z))$. Choose a basis $f_1,\dots,f_m$ of $\C[X]$ as a module over $\C[Y]$, and write $g=g_1f_1+\dots+g_mf_m$ for some $g_1,\dots,g_m\in\C[Y]$. Note that $m$ is the degree of the finite map $q$. Let $z\in Z$ be any point over which $q$ is unramified, and set $q^{-1}(z)=\{x_1,\dots,x_m\}$. Evaluating $g$ at $x_1,\dots,x_m$ and noting that $g(x_j)=0$ and $g_i(x_j)=g_i(z)$ for all $i,j$, we obtain equations
\[
0=g_1(z)f_1(x_j)+\dots+g_m(z)f_m(x_j)\;\;\;\text{ for  }\;\;\;j=1,\dots,m.
\]
In other words, the function
\[g_1(z)f_1 + \dots + g_m(z)f_m \in \C[X]\]
vanishes at all the preimages $x_1, \ldots, x_m$ of $z$, which means that it descends to zero in the quotient
\[\frac{\C[X]}{\mathcal{I}(z)\C[X]} \cong \left(\frac{\C[Y]}{\mathcal{I}(z)}\right)^m \cong \C^m.\]
Given that the images of $f_1, \ldots, f_m$ in this quotient form a basis as a $\C$-vector space, we conclude that $g_1(z) = \cdots = g_m(z) = 0$.  Since this holds for every $z\in Z$ over which $q$ is unramified, and since these points form a dense subset of $\overline Z$ by our hypotheses on $Z$, it then follows that $g_1,\dots,g_n\in\mathcal{I}(Z)$. The expression $g=g_1f_1+\dots+f_ng_n$ then implies that $g \in \mathcal{I}(Z)\C[X]$, as desired.
\end{proof}

At this point, we have successfully identified the wonderful compactification of the $r$-braid arrangement with the fiber product $\Mbar^r_n$ of a diagram in which three of the corners are moduli spaces.  It likely comes as no surprise, then, that $\Mbar^r_n$ has a modular interpretation, as well. We turn to this interpretation in the next and final section.

\section{The wonderful compactification as a moduli space}\label{sec:modulispace}

We begin this section by describing the moduli problem of interest, and we then prove that the fiber product $\Mbar^r_n$ is indeed a fine moduli space for this moduli problem.  Throughout, as above, we fix a primitive $r$th root of unity $\zeta$. Furthermore, all schemes are assumed to be finite-type over $\C$.

\subsection{The moduli problem}

Fix $r \geq 2$ and $n \geq 1$. We will refer to the objects in our moduli space as {\bf heavy $(r,n)$-curves}, to distinguish them from the $(r,n)$-curves studied in \cite{cdlr_permutohedral}; their definition is as follows.

\begin{definition}\label{def:rn_curve}
A {\bf heavy $(r,n)$-curve} is a nodal curve $C$ of arithmetic genus zero, equipped with $rn+2$ distinct marked points (denoted by $x^+$, $x^-$, and $z_i^j$ for $i \in [n]$ and $j \in \Z_r$) as well as an order-$r$ automorphism $\sigma: C \rightarrow C$, subject to the following conditions:
\begin{itemize}
  \item $\sigma(x^+) = x^+$ and $\sigma(x^-) = x^-$;
  \item $\sigma(z_i^j) = z_i^{j+1 \mod r}$ for each $i \in [n]$ and $j \in \Z_r$;
  \item $\sigma$ is generically free (that is, $\sigma$ has finite fixed locus);
  \item the $(rn+2)$-pointed curve $(C; \{z_i^j\}, x^{\pm})$ is stable (i.e., it is an element of $\Mbar_{0,rn+2}$).
\end{itemize}
These conditions ensure that $C$ consists of a central chain of $\P^1$'s connecting $x^+$ to $x^-$, to which the remaining trees of $\P^1$'s are attached with $r$-fold symmetry (see Figure~\ref{fig:rncurve}); to see this, note that any component invariant under the order-$r$ automorphism $\sigma$ has exactly two $\sigma$-fixed points, and it is only at these fixed points that the $\sigma$-invariant components can meet one another.  For each component $P \cong \P^1$ of the central chain, denote by $p^+$ and $p^-$ the $\sigma$-fixed points in the directions of $x^+$ and $x^-$, respectively.  In addition to the conditions above, we require that, in coordinates on $P$ in which $p^+ = \infty$ and $p^- = 0$, the automorphism $\sigma$ is given by multiplication by the fixed primitive root of unity $\zeta$.
\end{definition}

\tikzset{every picture/.style={line width=0.75pt}} %set default line width to 0.75pt    
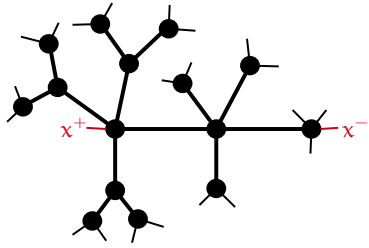
\begin{figure}[ht]
\begin{tikzpicture}[x=0.75pt,y=0.75pt,yscale=-1,xscale=1]

%Shape: Ellipse [id:dp9692824956607731] 
\draw [color={rgb, 255:red, 0; green, 0; blue, 0 } ,draw opacity=1 ][fill={rgb, 255:red, 0; green, 0; blue, 0 } ,fill opacity=1 ] (100.5,115.4) .. controls (100.5,112.97) and (102.51,111) .. (105,111) .. controls (107.49,111) and (109.5,112.97) .. (109.5,115.4) .. controls (109.5,117.83) and (107.49,119.8) .. (105,119.8) .. controls (102.51,119.8) and (100.5,117.83) .. (100.5,115.4) -- cycle ;
%Shape: Ellipse [id:dp9756630820469325] 
\draw [color={rgb, 255:red, 0; green, 0; blue, 0 } ,draw opacity=1 ][fill={rgb, 255:red, 0; green, 0; blue, 0 } ,fill opacity=1 ] (151.5,115.4) .. controls (151.5,112.97) and (153.51,111) .. (156,111) .. controls (158.49,111) and (160.5,112.97) .. (160.5,115.4) .. controls (160.5,117.83) and (158.49,119.8) .. (156,119.8) .. controls (153.51,119.8) and (151.5,117.83) .. (151.5,115.4) -- cycle ;
%Shape: Ellipse [id:dp10095612749594318] 
\draw [color={rgb, 255:red, 0; green, 0; blue, 0 } ,draw opacity=1 ][fill={rgb, 255:red, 0; green, 0; blue, 0 } ,fill opacity=1 ] (199.5,115.4) .. controls (199.5,112.97) and (201.51,111) .. (204,111) .. controls (206.49,111) and (208.5,112.97) .. (208.5,115.4) .. controls (208.5,117.83) and (206.49,119.8) .. (204,119.8) .. controls (201.51,119.8) and (199.5,117.83) .. (199.5,115.4) -- cycle ;
%Shape: Ellipse [id:dp18366158382014786] 
\draw [color={rgb, 255:red, 0; green, 0; blue, 0 } ,draw opacity=1 ][fill={rgb, 255:red, 0; green, 0; blue, 0 } ,fill opacity=1 ] (107.5,82.4) .. controls (107.5,79.97) and (109.51,78) .. (112,78) .. controls (114.49,78) and (116.5,79.97) .. (116.5,82.4) .. controls (116.5,84.83) and (114.49,86.8) .. (112,86.8) .. controls (109.51,86.8) and (107.5,84.83) .. (107.5,82.4) -- cycle ;
%Shape: Ellipse [id:dp6373415941623661] 
\draw [color={rgb, 255:red, 0; green, 0; blue, 0 } ,draw opacity=1 ][fill={rgb, 255:red, 0; green, 0; blue, 0 } ,fill opacity=1 ] (71.5,94.4) .. controls (71.5,91.97) and (73.51,90) .. (76,90) .. controls (78.49,90) and (80.5,91.97) .. (80.5,94.4) .. controls (80.5,96.83) and (78.49,98.8) .. (76,98.8) .. controls (73.51,98.8) and (71.5,96.83) .. (71.5,94.4) -- cycle ;
%Shape: Ellipse [id:dp36632427668925205] 
\draw [color={rgb, 255:red, 0; green, 0; blue, 0 } ,draw opacity=1 ][fill={rgb, 255:red, 0; green, 0; blue, 0 } ,fill opacity=1 ] (100.5,146.4) .. controls (100.5,143.97) and (102.51,142) .. (105,142) .. controls (107.49,142) and (109.5,143.97) .. (109.5,146.4) .. controls (109.5,148.83) and (107.49,150.8) .. (105,150.8) .. controls (102.51,150.8) and (100.5,148.83) .. (100.5,146.4) -- cycle ;
%Shape: Ellipse [id:dp7113337886413018] 
\draw [color={rgb, 255:red, 0; green, 0; blue, 0 } ,draw opacity=1 ][fill={rgb, 255:red, 0; green, 0; blue, 0 } ,fill opacity=1 ] (151.5,145.4) .. controls (151.5,142.97) and (153.51,141) .. (156,141) .. controls (158.49,141) and (160.5,142.97) .. (160.5,145.4) .. controls (160.5,147.83) and (158.49,149.8) .. (156,149.8) .. controls (153.51,149.8) and (151.5,147.83) .. (151.5,145.4) -- cycle ;
%Shape: Ellipse [id:dp7895171036251858] 
\draw [color={rgb, 255:red, 0; green, 0; blue, 0 } ,draw opacity=1 ][fill={rgb, 255:red, 0; green, 0; blue, 0 } ,fill opacity=1 ] (168.5,83.4) .. controls (168.5,80.97) and (170.51,79) .. (173,79) .. controls (175.49,79) and (177.5,80.97) .. (177.5,83.4) .. controls (177.5,85.83) and (175.49,87.8) .. (173,87.8) .. controls (170.51,87.8) and (168.5,85.83) .. (168.5,83.4) -- cycle ;
%Shape: Ellipse [id:dp8602917367906844] 
\draw [color={rgb, 255:red, 0; green, 0; blue, 0 } ,draw opacity=1 ][fill={rgb, 255:red, 0; green, 0; blue, 0 } ,fill opacity=1 ] (134.5,92.4) .. controls (134.5,89.97) and (136.51,88) .. (139,88) .. controls (141.49,88) and (143.5,89.97) .. (143.5,92.4) .. controls (143.5,94.83) and (141.49,96.8) .. (139,96.8) .. controls (136.51,96.8) and (134.5,94.83) .. (134.5,92.4) -- cycle ;
%Straight Lines [id:da5635943434902839] 
\draw [line width=1.5]  (76,94.4) -- (105,115.4) ;
%Straight Lines [id:da1942600813475952] 
\draw [line width=1.5]  (112,82.4) -- (105,115.4) ;
%Straight Lines [id:da3084729291510411] 
\draw [line width=1.5]  (105,115.4) -- (105,146.4) ;
%Straight Lines [id:da43172230694293723] 
\draw [line width=1.5]  (105,115.4) -- (156,115.4) ;
%Straight Lines [id:da12979161478814483] 
\draw [line width=1.5]  (156,115.4) -- (204,115.4) ;
%Straight Lines [id:da41480264489674834] 
\draw [line width=1.5]  (139,92.4) -- (156,115.4) ;
%Straight Lines [id:da4743560549100656] 
\draw [line width=1.5]  (173,83.4) -- (156,115.4) ;
%Straight Lines [id:da29492951499495734] 
\draw [line width=1.5]  (156,115.4) -- (156,145.4) ;
%Straight Lines [id:da9061930334385675] 
\draw [line width=1.5]  (128.5,66.8) -- (112,82.4) ;
%Straight Lines [id:da848623880918713] 
\draw [line width=1.5]  (112,82.4) -- (99.5,64.8) ;
%Straight Lines [id:da4902262251082292] 
\draw [line width=1.5]  (76,94.4) -- (60.5,102.8) ;
%Straight Lines [id:da5626243504035755] 
\draw [line width=1.5]  (76,94.4) -- (72.5,75.8) ;
%Straight Lines [id:da7040384022656294] 
\draw [line width=1.5]  (105,146.4) -- (116.5,160.8) ;
%Straight Lines [id:da31688038632733573] 
\draw [line width=1.5]  (105,146.4) -- (93.5,161.8) ;
%Shape: Ellipse [id:dp06956683103407157] 
\draw [color={rgb, 255:red, 0; green, 0; blue, 0 } ,draw opacity=1 ][fill={rgb, 255:red, 0; green, 0; blue, 0 } ,fill opacity=1 ] (112,160.8) .. controls (112,158.37) and (114.01,156.4) .. (116.5,156.4) .. controls (118.99,156.4) and (121,158.37) .. (121,160.8) .. controls (121,163.24) and (118.99,165.21) .. (116.5,165.21) .. controls (114.01,165.21) and (112,163.24) .. (112,160.8) -- cycle ;
%Shape: Ellipse [id:dp12363409309041207] 
\draw [color={rgb, 255:red, 0; green, 0; blue, 0 } ,draw opacity=1 ][fill={rgb, 255:red, 0; green, 0; blue, 0 } ,fill opacity=1 ] (89,161.8) .. controls (89,159.37) and (91.01,157.4) .. (93.5,157.4) .. controls (95.99,157.4) and (98,159.37) .. (98,161.8) .. controls (98,164.24) and (95.99,166.21) .. (93.5,166.21) .. controls (91.01,166.21) and (89,164.24) .. (89,161.8) -- cycle ;
%Shape: Ellipse [id:dp7361166282441527] 
\draw [color={rgb, 255:red, 0; green, 0; blue, 0 } ,draw opacity=1 ][fill={rgb, 255:red, 0; green, 0; blue, 0 } ,fill opacity=1 ] (54,104.21) .. controls (54,101.78) and (56.01,99.8) .. (58.5,99.8) .. controls (60.99,99.8) and (63,101.78) .. (63,104.21) .. controls (63,106.64) and (60.99,108.61) .. (58.5,108.61) .. controls (56.01,108.61) and (54,106.64) .. (54,104.21) -- cycle ;
%Shape: Ellipse [id:dp8424374471704033] 
\draw [color={rgb, 255:red, 0; green, 0; blue, 0 } ,draw opacity=1 ][fill={rgb, 255:red, 0; green, 0; blue, 0 } ,fill opacity=1 ] (67,72.4) .. controls (67,69.97) and (69.01,68) .. (71.5,68) .. controls (73.99,68) and (76,69.97) .. (76,72.4) .. controls (76,74.83) and (73.99,76.8) .. (71.5,76.8) .. controls (69.01,76.8) and (67,74.83) .. (67,72.4) -- cycle ;
%Shape: Ellipse [id:dp05024474372593901] 
\draw [color={rgb, 255:red, 0; green, 0; blue, 0 } ,draw opacity=1 ][fill={rgb, 255:red, 0; green, 0; blue, 0 } ,fill opacity=1 ] (93,62.4) .. controls (93,59.97) and (95.01,58) .. (97.5,58) .. controls (99.99,58) and (102,59.97) .. (102,62.4) .. controls (102,64.83) and (99.99,66.8) .. (97.5,66.8) .. controls (95.01,66.8) and (93,64.83) .. (93,62.4) -- cycle ;
%Shape: Ellipse [id:dp5269393516932115] 
\draw [color={rgb, 255:red, 0; green, 0; blue, 0 } ,draw opacity=1 ][fill={rgb, 255:red, 0; green, 0; blue, 0 } ,fill opacity=1 ] (127.5,64.8) .. controls (127.5,62.37) and (129.51,60.4) .. (132,60.4) .. controls (134.49,60.4) and (136.5,62.37) .. (136.5,64.8) .. controls (136.5,67.24) and (134.49,69.21) .. (132,69.21) .. controls (129.51,69.21) and (127.5,67.24) .. (127.5,64.8) -- cycle ;
%Straight Lines [id:da6141535639119629] 
\draw [color={rgb, 255:red, 208; green, 2; blue, 27 } ,draw opacity=1 ]  (208.5,115.4) -- (217.5,114.8) ;
%Straight Lines [id:da29736514480737464] 
\draw [color={rgb, 255:red, 208; green, 2; blue, 27 } ,draw opacity=1 ]  (90.5,114.8) -- (100.5,115.4) ;
%Straight Lines [id:da8489439377947259] 
\draw  (204,115.4) -- (203.5,127.8) ;
%Straight Lines [id:da846903570938714] 
\draw  (204,115.4) -- (211.5,105.8) ;
%Straight Lines [id:da41611043042349927] 
\draw  (204,115.4) -- (194.5,105.8) ;
%Straight Lines [id:da7314492420367797] 
\draw  (156,145.4) -- (165.5,154.8) ;
%Straight Lines [id:da8433377277816603] 
\draw  (156,145.4) -- (147.5,153.8) ;
%Straight Lines [id:da34574756280921703] 
\draw  (173,83.4) -- (187.5,83.8) ;
%Straight Lines [id:da07031317218347288] 
\draw  (173,83.4) -- (171.5,69.8) ;
%Straight Lines [id:da18507725274862352] 
\draw  (139,92.4) -- (143.5,81.8) ;
%Straight Lines [id:da9447009969703355] 
\draw  (139,92.4) -- (128.5,90.8) ;
%Straight Lines [id:da3860909778812813] 
\draw  (132,64.8) -- (145.5,65.8) ;
%Straight Lines [id:da8068156852728834] 
\draw  (132,64.8) -- (132.5,52.8) ;
%Straight Lines [id:da21878855680611786] 
\draw  (97.5,62.4) -- (103.5,51.8) ;
%Straight Lines [id:da273325927933741] 
\draw  (97.5,62.4) -- (83.5,62.8) ;
%Straight Lines [id:da09813033688372896] 
\draw  (116.5,160.8) -- (129.5,164.8) ;
%Straight Lines [id:da5767834328878381] 
\draw  (116.5,160.8) -- (113.5,172.8) ;
%Straight Lines [id:da9962088628177808] 
\draw  (93.5,161.8) -- (100.5,171.8) ;
%Straight Lines [id:da19622389179607413] 
\draw  (93.5,161.8) -- (83.5,168.8) ;
%Straight Lines [id:da5367782292946144] 
\draw  (58.5,104.21) -- (50.5,111.8) ;
%Straight Lines [id:da26348783741575854] 
\draw  (58.5,104.21) -- (54.5,93.8) ;
%Straight Lines [id:da5512086229556703] 
\draw  (71.5,72.4) -- (57.5,68.8) ;
%Straight Lines [id:da6482381796032617] 
\draw  (71.5,72.4) -- (76.5,61.8) ;

% Text Node
\draw (218,108.4) node [anchor=north west][inner sep=0.75pt] [font=\scriptsize,color={rgb, 255:red, 208; green, 2; blue, 27 } ,opacity=1 ] {$x^{-}$};
% Text Node
\draw (76,108.4) node [anchor=north west][inner sep=0.75pt] [font=\scriptsize,color={rgb, 255:red, 208; green, 2; blue, 27 } ,opacity=1 ] {$x^{+}$};
\end{tikzpicture}
\caption{A heavy $(r,n)$-curve with $r=3$ and $n=7$. The dual graph of the curve is shown, meaning that vertices represent components, edges represent nodes, and half-edges represent marked points. The automorphism $\sigma$ can be viewed as a rotation about the horizontal axis connecting the component with $x^+$ to the component with $x^-$.}
\label{fig:rncurve}
\end{figure}

Families of such objects are defined as follows.

\begin{definition}\label{def:family_rn_curve}
  A \emphbf{family of heavy $(r,n)$-curves over a base scheme $B$} is a flat morphism $\pi:\mcC \to B$ of schemes with $rn+2$ sections $z_i^j,x^+,x^-$ along with an order-$r$ automorphism $\sigma:\mcC \to \mcC$ satisfying $\pi \circ \sigma = \pi$ and such that the tuple
    \[
    \left(\pi^{-1}(b); \set{z_i^j(b)}, \; x^{\pm}(b), \; \sigma\vert_{\pi^{-1}(b)}\right)
    \]
    is a heavy $(r,n)$-curve for every $b\in B$. An isomorphism of families $\pi:\mcC\rightarrow B$ and $\pi':\mcC'\rightarrow B$ is an isomorphism $f:\mcC\rightarrow\mcC'$ over $B$ that commutes with the sections and the order-$r$ automorphisms.
\end{definition}

It is straightforward to see that families of heavy $(r,n)$-curves can be pulled back along morphisms $B' \rightarrow B$, so the above defines a moduli problem. It remains to show that the fiber product $\Mbar^r_n$ is a fine moduli space for this moduli problem. We accomplish this in several steps, first showing that the points of $\Mbar_n^r$ correspond to heavy $(r,n)$-curves, then constructing a tautological family over $\Mbar_n^r$, and finally by showing that this tautological family is universal.

\subsection{Points of $\Mbar_n^r$ parametrize heavy $(r,n)$-curves}

That points of $\Mbar_n^r$ parametrize heavy $(r,n)$-curves makes sense heuristically: the data of a heavy $(r,n)$-curve is dictated by the choice of the quotient $C/\sigma$ (which is an element of $\Mbar_{0,n+2}$) together with the choice, for each $i \in [n]$, of which of the points $z_i^j$ is $z_i^0$ (which is given by an element of $\Lbar_n$ whose image under the $r$th power map $q$ is the contraction of $C/\sigma$ to its central chain). See Figure~\ref{fig:fiberproduct}. The two lemmas of this subsection make this heuristic more precise.

\tikzset{every picture/.style={line width=0.75pt}} %set default line width to 0.75pt    

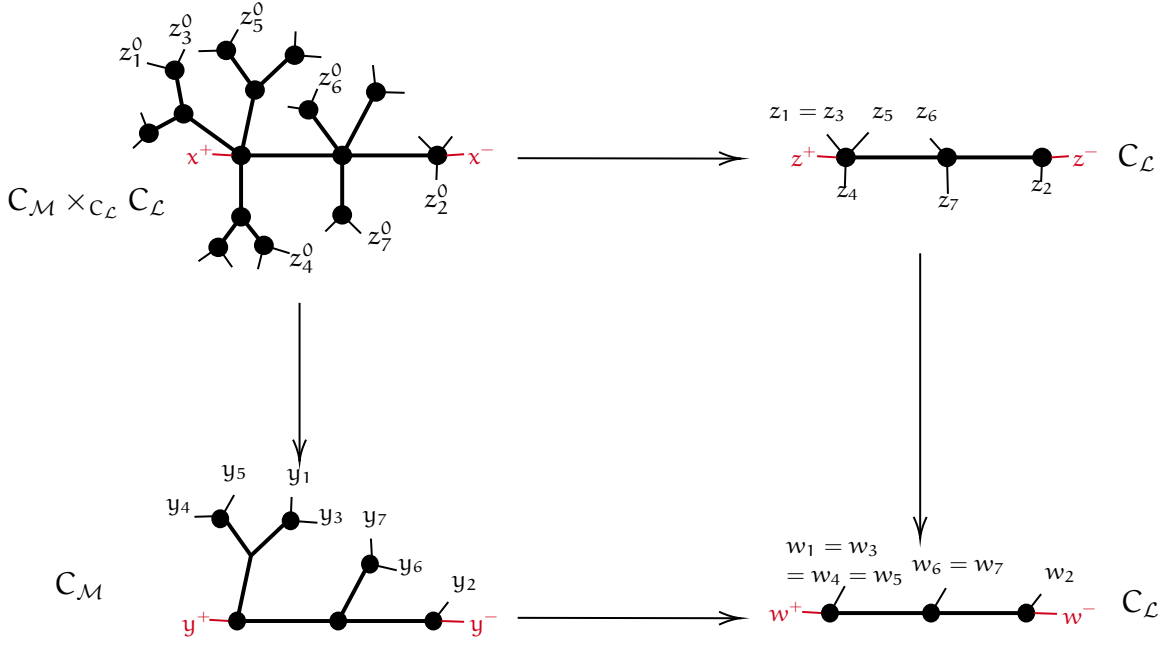
\begin{figure}[ht]
\begin{tikzpicture}[x=0.75pt,y=0.75pt,yscale=-1,xscale=1]
%uncomment if require: \path (0,391); %set diagram left start at 0, and has height of 391

%Shape: Ellipse [id:dp9692824956607731] 
\draw [color={rgb, 255:red, 0; green, 0; blue, 0 } ,draw opacity=1 ][fill={rgb, 255:red, 0; green, 0; blue, 0 } ,fill opacity=1 ] (100.5,115.4) .. controls (100.5,112.97) and (102.51,111) .. (105,111) .. controls (107.49,111) and (109.5,112.97) .. (109.5,115.4) .. controls (109.5,117.83) and (107.49,119.8) .. (105,119.8) .. controls (102.51,119.8) and (100.5,117.83) .. (100.5,115.4) -- cycle ;
%Shape: Ellipse [id:dp9756630820469325] 
\draw [color={rgb, 255:red, 0; green, 0; blue, 0 } ,draw opacity=1 ][fill={rgb, 255:red, 0; green, 0; blue, 0 } ,fill opacity=1 ] (151.5,115.4) .. controls (151.5,112.97) and (153.51,111) .. (156,111) .. controls (158.49,111) and (160.5,112.97) .. (160.5,115.4) .. controls (160.5,117.83) and (158.49,119.8) .. (156,119.8) .. controls (153.51,119.8) and (151.5,117.83) .. (151.5,115.4) -- cycle ;
%Shape: Ellipse [id:dp10095612749594318] 
\draw [color={rgb, 255:red, 0; green, 0; blue, 0 } ,draw opacity=1 ][fill={rgb, 255:red, 0; green, 0; blue, 0 } ,fill opacity=1 ] (199.5,115.4) .. controls (199.5,112.97) and (201.51,111) .. (204,111) .. controls (206.49,111) and (208.5,112.97) .. (208.5,115.4) .. controls (208.5,117.83) and (206.49,119.8) .. (204,119.8) .. controls (201.51,119.8) and (199.5,117.83) .. (199.5,115.4) -- cycle ;
%Shape: Ellipse [id:dp18366158382014786] 
\draw [color={rgb, 255:red, 0; green, 0; blue, 0 } ,draw opacity=1 ][fill={rgb, 255:red, 0; green, 0; blue, 0 } ,fill opacity=1 ] (107.5,82.4) .. controls (107.5,79.97) and (109.51,78) .. (112,78) .. controls (114.49,78) and (116.5,79.97) .. (116.5,82.4) .. controls (116.5,84.83) and (114.49,86.8) .. (112,86.8) .. controls (109.51,86.8) and (107.5,84.83) .. (107.5,82.4) -- cycle ;
%Shape: Ellipse [id:dp6373415941623661] 
\draw [color={rgb, 255:red, 0; green, 0; blue, 0 } ,draw opacity=1 ][fill={rgb, 255:red, 0; green, 0; blue, 0 } ,fill opacity=1 ] (71.5,94.4) .. controls (71.5,91.97) and (73.51,90) .. (76,90) .. controls (78.49,90) and (80.5,91.97) .. (80.5,94.4) .. controls (80.5,96.83) and (78.49,98.8) .. (76,98.8) .. controls (73.51,98.8) and (71.5,96.83) .. (71.5,94.4) -- cycle ;
%Shape: Ellipse [id:dp36632427668925205] 
\draw [color={rgb, 255:red, 0; green, 0; blue, 0 } ,draw opacity=1 ][fill={rgb, 255:red, 0; green, 0; blue, 0 } ,fill opacity=1 ] (100.5,146.4) .. controls (100.5,143.97) and (102.51,142) .. (105,142) .. controls (107.49,142) and (109.5,143.97) .. (109.5,146.4) .. controls (109.5,148.83) and (107.49,150.8) .. (105,150.8) .. controls (102.51,150.8) and (100.5,148.83) .. (100.5,146.4) -- cycle ;
%Shape: Ellipse [id:dp7113337886413018] 
\draw [color={rgb, 255:red, 0; green, 0; blue, 0 } ,draw opacity=1 ][fill={rgb, 255:red, 0; green, 0; blue, 0 } ,fill opacity=1 ] (151.5,145.4) .. controls (151.5,142.97) and (153.51,141) .. (156,141) .. controls (158.49,141) and (160.5,142.97) .. (160.5,145.4) .. controls (160.5,147.83) and (158.49,149.8) .. (156,149.8) .. controls (153.51,149.8) and (151.5,147.83) .. (151.5,145.4) -- cycle ;
%Shape: Ellipse [id:dp7895171036251858] 
\draw [color={rgb, 255:red, 0; green, 0; blue, 0 } ,draw opacity=1 ][fill={rgb, 255:red, 0; green, 0; blue, 0 } ,fill opacity=1 ] (168.5,83.4) .. controls (168.5,80.97) and (170.51,79) .. (173,79) .. controls (175.49,79) and (177.5,80.97) .. (177.5,83.4) .. controls (177.5,85.83) and (175.49,87.8) .. (173,87.8) .. controls (170.51,87.8) and (168.5,85.83) .. (168.5,83.4) -- cycle ;
%Shape: Ellipse [id:dp8602917367906844] 
\draw [color={rgb, 255:red, 0; green, 0; blue, 0 } ,draw opacity=1 ][fill={rgb, 255:red, 0; green, 0; blue, 0 } ,fill opacity=1 ] (134.5,92.4) .. controls (134.5,89.97) and (136.51,88) .. (139,88) .. controls (141.49,88) and (143.5,89.97) .. (143.5,92.4) .. controls (143.5,94.83) and (141.49,96.8) .. (139,96.8) .. controls (136.51,96.8) and (134.5,94.83) .. (134.5,92.4) -- cycle ;
%Straight Lines [id:da5635943434902839] 
\draw [line width=1.5]  (76,94.4) -- (105,115.4) ;
%Straight Lines [id:da1942600813475952] 
\draw [line width=1.5]  (112,82.4) -- (105,115.4) ;
%Straight Lines [id:da3084729291510411] 
\draw [line width=1.5]  (105,115.4) -- (105,146.4) ;
%Straight Lines [id:da43172230694293723] 
\draw [line width=1.5]  (105,115.4) -- (156,115.4) ;
%Straight Lines [id:da12979161478814483] 
\draw [line width=1.5]  (156,115.4) -- (204,115.4) ;
%Straight Lines [id:da41480264489674834] 
\draw [line width=1.5]  (139,92.4) -- (156,115.4) ;
%Straight Lines [id:da4743560549100656] 
\draw [line width=1.5]  (173,83.4) -- (156,115.4) ;
%Straight Lines [id:da29492951499495734] 
\draw [line width=1.5]  (156,115.4) -- (156,145.4) ;
%Straight Lines [id:da9061930334385675] 
\draw [line width=1.5]  (128.5,66.8) -- (112,82.4) ;
%Straight Lines [id:da848623880918713] 
\draw [line width=1.5]  (112,82.4) -- (99.5,64.8) ;
%Straight Lines [id:da4902262251082292] 
\draw [line width=1.5]  (76,94.4) -- (60.5,102.8) ;
%Straight Lines [id:da5626243504035755] 
\draw [line width=1.5]  (76,94.4) -- (72.5,75.8) ;
%Straight Lines [id:da7040384022656294] 
\draw [line width=1.5]  (105,146.4) -- (116.5,160.8) ;
%Straight Lines [id:da31688038632733573] 
\draw [line width=1.5]  (105,146.4) -- (93.5,161.8) ;
%Shape: Ellipse [id:dp06956683103407157] 
\draw [color={rgb, 255:red, 0; green, 0; blue, 0 } ,draw opacity=1 ][fill={rgb, 255:red, 0; green, 0; blue, 0 } ,fill opacity=1 ] (112,160.8) .. controls (112,158.37) and (114.01,156.4) .. (116.5,156.4) .. controls (118.99,156.4) and (121,158.37) .. (121,160.8) .. controls (121,163.24) and (118.99,165.21) .. (116.5,165.21) .. controls (114.01,165.21) and (112,163.24) .. (112,160.8) -- cycle ;
%Shape: Ellipse [id:dp12363409309041207] 
\draw [color={rgb, 255:red, 0; green, 0; blue, 0 } ,draw opacity=1 ][fill={rgb, 255:red, 0; green, 0; blue, 0 } ,fill opacity=1 ] (89,161.8) .. controls (89,159.37) and (91.01,157.4) .. (93.5,157.4) .. controls (95.99,157.4) and (98,159.37) .. (98,161.8) .. controls (98,164.24) and (95.99,166.21) .. (93.5,166.21) .. controls (91.01,166.21) and (89,164.24) .. (89,161.8) -- cycle ;
%Shape: Ellipse [id:dp7361166282441527] 
\draw [color={rgb, 255:red, 0; green, 0; blue, 0 } ,draw opacity=1 ][fill={rgb, 255:red, 0; green, 0; blue, 0 } ,fill opacity=1 ] (54,104.21) .. controls (54,101.78) and (56.01,99.8) .. (58.5,99.8) .. controls (60.99,99.8) and (63,101.78) .. (63,104.21) .. controls (63,106.64) and (60.99,108.61) .. (58.5,108.61) .. controls (56.01,108.61) and (54,106.64) .. (54,104.21) -- cycle ;
%Shape: Ellipse [id:dp8424374471704033] 
\draw [color={rgb, 255:red, 0; green, 0; blue, 0 } ,draw opacity=1 ][fill={rgb, 255:red, 0; green, 0; blue, 0 } ,fill opacity=1 ] (67,72.4) .. controls (67,69.97) and (69.01,68) .. (71.5,68) .. controls (73.99,68) and (76,69.97) .. (76,72.4) .. controls (76,74.83) and (73.99,76.8) .. (71.5,76.8) .. controls (69.01,76.8) and (67,74.83) .. (67,72.4) -- cycle ;
%Shape: Ellipse [id:dp05024474372593901] 
\draw [color={rgb, 255:red, 0; green, 0; blue, 0 } ,draw opacity=1 ][fill={rgb, 255:red, 0; green, 0; blue, 0 } ,fill opacity=1 ] (93,62.4) .. controls (93,59.97) and (95.01,58) .. (97.5,58) .. controls (99.99,58) and (102,59.97) .. (102,62.4) .. controls (102,64.83) and (99.99,66.8) .. (97.5,66.8) .. controls (95.01,66.8) and (93,64.83) .. (93,62.4) -- cycle ;
%Shape: Ellipse [id:dp5269393516932115] 
\draw [color={rgb, 255:red, 0; green, 0; blue, 0 } ,draw opacity=1 ][fill={rgb, 255:red, 0; green, 0; blue, 0 } ,fill opacity=1 ] (127.5,64.8) .. controls (127.5,62.37) and (129.51,60.4) .. (132,60.4) .. controls (134.49,60.4) and (136.5,62.37) .. (136.5,64.8) .. controls (136.5,67.24) and (134.49,69.21) .. (132,69.21) .. controls (129.51,69.21) and (127.5,67.24) .. (127.5,64.8) -- cycle ;
%Straight Lines [id:da6141535639119629] 
\draw [color={rgb, 255:red, 208; green, 2; blue, 27 } ,draw opacity=1 ]  (208.5,115.4) -- (217.5,114.8) ;
%Straight Lines [id:da29736514480737464] 
\draw [color={rgb, 255:red, 208; green, 2; blue, 27 } ,draw opacity=1 ]  (90.5,114.8) -- (100.5,115.4) ;
%Straight Lines [id:da8489439377947259] 
\draw  (204,115.4) -- (203.5,127.8) ;
%Straight Lines [id:da846903570938714] 
\draw  (204,115.4) -- (211.5,105.8) ;
%Straight Lines [id:da41611043042349927] 
\draw  (204,115.4) -- (194.5,105.8) ;
%Straight Lines [id:da7314492420367797] 
\draw  (156,145.4) -- (165.5,154.8) ;
%Straight Lines [id:da8433377277816603] 
\draw  (156,145.4) -- (147.5,153.8) ;
%Straight Lines [id:da34574756280921703] 
\draw  (173,83.4) -- (187.5,83.8) ;
%Straight Lines [id:da07031317218347288] 
\draw  (173,83.4) -- (171.5,69.8) ;
%Straight Lines [id:da18507725274862352] 
\draw  (139,92.4) -- (143.5,81.8) ;
%Straight Lines [id:da9447009969703355] 
\draw  (139,92.4) -- (128.5,90.8) ;
%Straight Lines [id:da3860909778812813] 
\draw  (132,64.8) -- (145.5,65.8) ;
%Straight Lines [id:da8068156852728834] 
\draw  (132,64.8) -- (132.5,52.8) ;
%Straight Lines [id:da21878855680611786] 
\draw  (97.5,62.4) -- (103.5,51.8) ;
%Straight Lines [id:da273325927933741] 
\draw  (97.5,62.4) -- (83.5,62.8) ;
%Straight Lines [id:da09813033688372896] 
\draw  (116.5,160.8) -- (129.5,164.8) ;
%Straight Lines [id:da5767834328878381] 
\draw  (116.5,160.8) -- (113.5,172.8) ;
%Straight Lines [id:da9962088628177808] 
\draw  (93.5,161.8) -- (100.5,171.8) ;
%Straight Lines [id:da19622389179607413] 
\draw  (93.5,161.8) -- (83.5,168.8) ;
%Straight Lines [id:da5367782292946144] 
\draw  (58.5,104.21) -- (50.5,111.8) ;
%Straight Lines [id:da26348783741575854] 
\draw  (58.5,104.21) -- (54.5,93.8) ;
%Straight Lines [id:da5512086229556703] 
\draw  (71.5,72.4) -- (57.5,68.8) ;
%Straight Lines [id:da6482381796032617] 
\draw  (71.5,72.4) -- (76.5,61.8) ;
%Straight Lines [id:da23454976573363662] 
\draw [line width=1.5]  (110,317.4) -- (103,350.4) ;
%Straight Lines [id:da3367460944351037] 
\draw [line width=1.5]  (103,350.4) -- (154,350.4) ;
%Straight Lines [id:da9992368919524903] 
\draw [line width=1.5]  (154,350.4) -- (202,350.4) ;
%Straight Lines [id:da10611220580812719] 
\draw [line width=1.5]  (171,318.4) -- (154,350.4) ;
%Straight Lines [id:da47103221549006236] 
\draw [line width=1.5]  (126.5,301.8) -- (110,317.4) ;
%Straight Lines [id:da08651858083907937] 
\draw [line width=1.5]  (110,317.4) -- (97.5,299.8) ;
%Straight Lines [id:da6763683276181711] 
\draw [color={rgb, 255:red, 208; green, 2; blue, 27 } ,draw opacity=1 ]  (89,349.8) -- (99,350.4) ;
%Straight Lines [id:da5874344321822107] 
\draw  (202,350.4) -- (209.5,340.8) ;
%Straight Lines [id:da18874419251725016] 
\draw  (171,318.4) -- (170.5,308.8) ;
%Straight Lines [id:da8765484422160186] 
\draw  (130,299.8) -- (143.5,300.8) ;
%Straight Lines [id:da9455387751700994] 
\draw  (130,299.8) -- (130.5,287.8) ;
%Straight Lines [id:da0755610471758228] 
\draw  (95.5,297.4) -- (101.5,286.8) ;
%Straight Lines [id:da2777678548638365] 
\draw  (95.5,297.4) -- (81.5,297.8) ;
%Shape: Ellipse [id:dp7312826700410384] 
\draw [color={rgb, 255:red, 0; green, 0; blue, 0 } ,draw opacity=1 ][fill={rgb, 255:red, 0; green, 0; blue, 0 } ,fill opacity=1 ] (99,350.4) .. controls (99,348.04) and (100.79,346.13) .. (103,346.13) .. controls (105.21,346.13) and (107,348.04) .. (107,350.4) .. controls (107,352.76) and (105.21,354.67) .. (103,354.67) .. controls (100.79,354.67) and (99,352.76) .. (99,350.4) -- cycle ;
%Shape: Ellipse [id:dp08047719647869633] 
\draw [color={rgb, 255:red, 0; green, 0; blue, 0 } ,draw opacity=1 ][fill={rgb, 255:red, 0; green, 0; blue, 0 } ,fill opacity=1 ] (90.5,298.8) .. controls (90.5,296.45) and (92.29,294.53) .. (94.5,294.53) .. controls (96.71,294.53) and (98.5,296.45) .. (98.5,298.8) .. controls (98.5,301.16) and (96.71,303.08) .. (94.5,303.08) .. controls (92.29,303.08) and (90.5,301.16) .. (90.5,298.8) -- cycle ;
%Shape: Ellipse [id:dp6802627519835551] 
\draw [color={rgb, 255:red, 0; green, 0; blue, 0 } ,draw opacity=1 ][fill={rgb, 255:red, 0; green, 0; blue, 0 } ,fill opacity=1 ] (126,299.8) .. controls (126,297.45) and (127.79,295.53) .. (130,295.53) .. controls (132.21,295.53) and (134,297.45) .. (134,299.8) .. controls (134,302.16) and (132.21,304.08) .. (130,304.08) .. controls (127.79,304.08) and (126,302.16) .. (126,299.8) -- cycle ;
%Shape: Ellipse [id:dp8860239561107579] 
\draw [color={rgb, 255:red, 0; green, 0; blue, 0 } ,draw opacity=1 ][fill={rgb, 255:red, 0; green, 0; blue, 0 } ,fill opacity=1 ] (150,350.4) .. controls (150,348.04) and (151.79,346.13) .. (154,346.13) .. controls (156.21,346.13) and (158,348.04) .. (158,350.4) .. controls (158,352.76) and (156.21,354.67) .. (154,354.67) .. controls (151.79,354.67) and (150,352.76) .. (150,350.4) -- cycle ;
%Shape: Ellipse [id:dp47462305722973475] 
\draw [color={rgb, 255:red, 0; green, 0; blue, 0 } ,draw opacity=1 ][fill={rgb, 255:red, 0; green, 0; blue, 0 } ,fill opacity=1 ] (166,321.67) .. controls (166,319.31) and (167.79,317.4) .. (170,317.4) .. controls (172.21,317.4) and (174,319.31) .. (174,321.67) .. controls (174,324.03) and (172.21,325.95) .. (170,325.95) .. controls (167.79,325.95) and (166,324.03) .. (166,321.67) -- cycle ;
%Shape: Ellipse [id:dp6301477706544325] 
\draw [color={rgb, 255:red, 0; green, 0; blue, 0 } ,draw opacity=1 ][fill={rgb, 255:red, 0; green, 0; blue, 0 } ,fill opacity=1 ] (198,350.4) .. controls (198,348.04) and (199.79,346.13) .. (202,346.13) .. controls (204.21,346.13) and (206,348.04) .. (206,350.4) .. controls (206,352.76) and (204.21,354.67) .. (202,354.67) .. controls (199.79,354.67) and (198,352.76) .. (198,350.4) -- cycle ;
%Straight Lines [id:da6100231113288259] 
\draw  (170,321.67) -- (183.5,324.8) ;
%Straight Lines [id:da9131950029018177] 
\draw [color={rgb, 255:red, 208; green, 2; blue, 27 } ,draw opacity=1 ]  (206,350.4) -- (217.5,350.8) ;
%Shape: Ellipse [id:dp6741663670865745] 
\draw [color={rgb, 255:red, 0; green, 0; blue, 0 } ,draw opacity=1 ][fill={rgb, 255:red, 0; green, 0; blue, 0 } ,fill opacity=1 ] (405.5,116.4) .. controls (405.5,113.97) and (407.51,112) .. (410,112) .. controls (412.49,112) and (414.5,113.97) .. (414.5,116.4) .. controls (414.5,118.83) and (412.49,120.8) .. (410,120.8) .. controls (407.51,120.8) and (405.5,118.83) .. (405.5,116.4) -- cycle ;
%Shape: Ellipse [id:dp9386773955364585] 
\draw [color={rgb, 255:red, 0; green, 0; blue, 0 } ,draw opacity=1 ][fill={rgb, 255:red, 0; green, 0; blue, 0 } ,fill opacity=1 ] (456.5,116.4) .. controls (456.5,113.97) and (458.51,112) .. (461,112) .. controls (463.49,112) and (465.5,113.97) .. (465.5,116.4) .. controls (465.5,118.83) and (463.49,120.8) .. (461,120.8) .. controls (458.51,120.8) and (456.5,118.83) .. (456.5,116.4) -- cycle ;
%Shape: Ellipse [id:dp5328373707364613] 
\draw [color={rgb, 255:red, 0; green, 0; blue, 0 } ,draw opacity=1 ][fill={rgb, 255:red, 0; green, 0; blue, 0 } ,fill opacity=1 ] (504.5,116.4) .. controls (504.5,113.97) and (506.51,112) .. (509,112) .. controls (511.49,112) and (513.5,113.97) .. (513.5,116.4) .. controls (513.5,118.83) and (511.49,120.8) .. (509,120.8) .. controls (506.51,120.8) and (504.5,118.83) .. (504.5,116.4) -- cycle ;
%Straight Lines [id:da07554255688154932] 
\draw [line width=1.5]  (410,116.4) -- (461,116.4) ;
%Straight Lines [id:da1891228965000218] 
\draw [line width=1.5]  (461,116.4) -- (509,116.4) ;
%Straight Lines [id:da9226039174232619] 
\draw [color={rgb, 255:red, 208; green, 2; blue, 27 } ,draw opacity=1 ]  (513.5,116.4) -- (522.5,115.8) ;
%Straight Lines [id:da09756150905649341] 
\draw [color={rgb, 255:red, 208; green, 2; blue, 27 } ,draw opacity=1 ]  (395.5,115.8) -- (405.5,116.4) ;
%Straight Lines [id:da8659026815980264] 
\draw  (509,116.4) -- (508.5,128.8) ;
%Straight Lines [id:da8415433022582874] 
\draw  (400.5,104.8) -- (410,116.4) ;
%Straight Lines [id:da33995535254561604] 
\draw  (421.5,103.8) -- (410,116.4) ;
%Straight Lines [id:da9321787863869111] 
\draw  (409.5,131.8) -- (410,116.4) ;
%Straight Lines [id:da9834892846232322] 
\draw  (452.5,106.8) -- (461,116.4) ;
%Straight Lines [id:da23788220850401887] 
\draw  (461.5,134.8) -- (461,116.4) ;
%Straight Lines [id:da9597985384566743] 
\draw [line width=1.5]  (402,346.4) -- (453,346.4) ;
%Straight Lines [id:da9310211733544009] 
\draw [line width=1.5]  (453,346.4) -- (501,346.4) ;
%Straight Lines [id:da2730632981146701] 
\draw [color={rgb, 255:red, 208; green, 2; blue, 27 } ,draw opacity=1 ]  (388,345.8) -- (398,346.4) ;
%Straight Lines [id:da9471629527321075] 
\draw  (501,346.4) -- (508.5,336.8) ;
%Shape: Ellipse [id:dp38163101771950214] 
\draw [color={rgb, 255:red, 0; green, 0; blue, 0 } ,draw opacity=1 ][fill={rgb, 255:red, 0; green, 0; blue, 0 } ,fill opacity=1 ] (398,346.4) .. controls (398,344.04) and (399.79,342.13) .. (402,342.13) .. controls (404.21,342.13) and (406,344.04) .. (406,346.4) .. controls (406,348.76) and (404.21,350.67) .. (402,350.67) .. controls (399.79,350.67) and (398,348.76) .. (398,346.4) -- cycle ;
%Shape: Ellipse [id:dp2428724208515547] 
\draw [color={rgb, 255:red, 0; green, 0; blue, 0 } ,draw opacity=1 ][fill={rgb, 255:red, 0; green, 0; blue, 0 } ,fill opacity=1 ] (449,346.4) .. controls (449,344.04) and (450.79,342.13) .. (453,342.13) .. controls (455.21,342.13) and (457,344.04) .. (457,346.4) .. controls (457,348.76) and (455.21,350.67) .. (453,350.67) .. controls (450.79,350.67) and (449,348.76) .. (449,346.4) -- cycle ;
%Shape: Ellipse [id:dp36027293832469365] 
\draw [color={rgb, 255:red, 0; green, 0; blue, 0 } ,draw opacity=1 ][fill={rgb, 255:red, 0; green, 0; blue, 0 } ,fill opacity=1 ] (497,346.4) .. controls (497,344.04) and (498.79,342.13) .. (501,342.13) .. controls (503.21,342.13) and (505,344.04) .. (505,346.4) .. controls (505,348.76) and (503.21,350.67) .. (501,350.67) .. controls (498.79,350.67) and (497,348.76) .. (497,346.4) -- cycle ;
%Straight Lines [id:da2545455032690298] 
\draw [color={rgb, 255:red, 208; green, 2; blue, 27 } ,draw opacity=1 ]  (505,346.4) -- (516.5,346.8) ;
%Straight Lines [id:da5682779018935653] 
\draw  (460.5,333.8) -- (453,346.4) ;
%Straight Lines [id:da6838062666974789] 
\draw  (409.5,332.8) -- (402,346.4) ;
%Straight Lines [id:da8303778855776793] 
\draw  (244.5,348.8) -- (355,349) ;
\draw [shift={(357,349)}, rotate = 180.1] [color={rgb, 255:red, 0; green, 0; blue, 0 } ][line width=0.75]  (10.93,-3.29) .. controls (6.95,-1.4) and (3.31,-0.3) .. (0,0) .. controls (3.31,0.3) and (6.95,1.4) .. (10.93,3.29)  ;
%Straight Lines [id:da03001671365768588] 
\draw  (244.5,116.8) -- (355,117) ;
\draw [shift={(357,117)}, rotate = 180.1] [color={rgb, 255:red, 0; green, 0; blue, 0 } ][line width=0.75]  (10.93,-3.29) .. controls (6.95,-1.4) and (3.31,-0.3) .. (0,0) .. controls (3.31,0.3) and (6.95,1.4) .. (10.93,3.29)  ;
%Straight Lines [id:da7462084553168917] 
\draw  (447.5,164.8) -- (447.5,306.8) ;
\draw [shift={(447.5,308.8)}, rotate = 270] [color={rgb, 255:red, 0; green, 0; blue, 0 } ][line width=0.75]  (10.93,-3.29) .. controls (6.95,-1.4) and (3.31,-0.3) .. (0,0) .. controls (3.31,0.3) and (6.95,1.4) .. (10.93,3.29)  ;
%Straight Lines [id:da464763131320604] 
\draw  (134.5,189.8) -- (134.5,267.8) ;
\draw [shift={(134.5,269.8)}, rotate = 270] [color={rgb, 255:red, 0; green, 0; blue, 0 } ][line width=0.75]  (10.93,-3.29) .. controls (6.95,-1.4) and (3.31,-0.3) .. (0,0) .. controls (3.31,0.3) and (6.95,1.4) .. (10.93,3.29)  ;

% Text Node
\draw (218,108.4) node [anchor=north west][inner sep=0.75pt] [font=\scriptsize,color={rgb, 255:red, 208; green, 2; blue, 27 } ,opacity=1 ] {$x^{-}$};
% Text Node
\draw (76,108.4) node [anchor=north west][inner sep=0.75pt] [font=\scriptsize,color={rgb, 255:red, 208; green, 2; blue, 27 } ,opacity=1 ] {$x^{+}$};
% Text Node
\draw (74,344.4) node [anchor=north west][inner sep=0.75pt] [font=\scriptsize,color={rgb, 255:red, 208; green, 2; blue, 27 } ,opacity=1 ] {$y^{+}$};
% Text Node
\draw (219,344.4) node [anchor=north west][inner sep=0.75pt] [font=\scriptsize,color={rgb, 255:red, 208; green, 2; blue, 27 } ,opacity=1 ] {$y^{-}$};
% Text Node
\draw (523,109.4) node [anchor=north west][inner sep=0.75pt] [font=\scriptsize,color={rgb, 255:red, 208; green, 2; blue, 27 } ,opacity=1 ] {$z^{-}$};
% Text Node
\draw (381,109.4) node [anchor=north west][inner sep=0.75pt] [font=\scriptsize,color={rgb, 255:red, 208; green, 2; blue, 27 } ,opacity=1 ] {$z^{+}$};
% Text Node
\draw (196,128.4) node [anchor=north west][inner sep=0.75pt] [font=\footnotesize] {$z_{2}^{0}$};
% Text Node
\draw (42,54.4) node [anchor=north west][inner sep=0.75pt] [font=\footnotesize] {$z_{1}^{0}$};
% Text Node
\draw (502,126.4) node [anchor=north west][inner sep=0.75pt] [font=\scriptsize] {$z_{2}$};
% Text Node
\draw (370,89.4) node [anchor=north west][inner sep=0.75pt] [font=\scriptsize] {$z_{1} =z_{3}$};
% Text Node
\draw (67,42.4) node [anchor=north west][inner sep=0.75pt] [font=\footnotesize] {$z_{3}^{0}$};
% Text Node
\draw (103,36.4) node [anchor=north west][inner sep=0.75pt] [font=\footnotesize] {$z_{5}^{0}$};
% Text Node
\draw (422,89.4) node [anchor=north west][inner sep=0.75pt] [font=\scriptsize] {$z_{5}$};
% Text Node
\draw (128,158.4) node [anchor=north west][inner sep=0.75pt] [font=\footnotesize] {$z_{4}^{0}$};
% Text Node
\draw (404,129.4) node [anchor=north west][inner sep=0.75pt] [font=\scriptsize] {$z_{4}$};
% Text Node
\draw (444,89.4) node [anchor=north west][inner sep=0.75pt] [font=\scriptsize] {$z_{6}$};
% Text Node
\draw (455,134.4) node [anchor=north west][inner sep=0.75pt] [font=\scriptsize] {$z_{7}$};
% Text Node
\draw (210,326.4) node [anchor=north west][inner sep=0.75pt] [font=\scriptsize] {$y_{2}$};
% Text Node
\draw (142.5,67.2) node [anchor=north west][inner sep=0.75pt] [font=\footnotesize] {$z_{6}^{0}$};
% Text Node
\draw (167.5,147.2) node [anchor=north west][inner sep=0.75pt] [font=\footnotesize] {$z_{7}^{0}$};
% Text Node
\draw (183,317.4) node [anchor=north west][inner sep=0.75pt] [font=\scriptsize] {$y_{6}$};
% Text Node
\draw (165,293.4) node [anchor=north west][inner sep=0.75pt] [font=\scriptsize] {$y_{7}$};
% Text Node
\draw (66,287.4) node [anchor=north west][inner sep=0.75pt] [font=\scriptsize] {$y_{4}$};
% Text Node
\draw (94,270.4) node [anchor=north west][inner sep=0.75pt] [font=\scriptsize] {$y_{5}$};
% Text Node
\draw (127,272.4) node [anchor=north west][inner sep=0.75pt] [font=\scriptsize] {$y_{1}$};
% Text Node
\draw (143,292.4) node [anchor=north west][inner sep=0.75pt] [font=\scriptsize] {$y_{3}$};
% Text Node
\draw (370,339) node [anchor=north west][inner sep=0.75pt] [font=\scriptsize,color={rgb, 255:red, 208; green, 2; blue, 27 } ,opacity=1 ] {$w^{+}$};
% Text Node
\draw (518,340.4) node [anchor=north west][inner sep=0.75pt] [font=\scriptsize,color={rgb, 255:red, 208; green, 2; blue, 27 } ,opacity=1 ] {$w^{-}$};
% Text Node
\draw (509,322.4) node [anchor=north west][inner sep=0.75pt] [font=\scriptsize] {$w_{2}$};
% Text Node
\draw (442,318.4) node [anchor=north west][inner sep=0.75pt] [font=\scriptsize] {$w_{6} =w_{7}$};
% Text Node
\draw (372,304.4) node [anchor=north west][inner sep=0.75pt] [font=\scriptsize] {$ \begin{array}{l}
w_{1} =w_{3}\\
=w_{4} =w_{5}
\end{array}$};
% Text Node
\draw (9,324.4) node [anchor=north west][inner sep=0.75pt]  {$C_\M$};
% Text Node
\draw (-15,130.4) node [anchor=north west][inner sep=0.75pt]  {$C_\M\times_{C_\mcL} C_\mcL$};
% Text Node
\draw (547,333.4) node [anchor=north west][inner sep=0.75pt]  {$C_\mcL$};
% Text Node
\draw (545,108.4) node [anchor=north west][inner sep=0.75pt]  {$C_\mcL$};
\end{tikzpicture}
\caption{An illustration of the interpretation of $\Mbar^r_n$ as a fiber product. Note that we have not labeled the points $z_i^j$ for $j\neq 0$ in the upper-left diagram, but those labels are uniquely determined by $z_i^0$ and $\sigma$.}
\label{fig:fiberproduct}
\end{figure}

Since $\Mbar_n^r=\Mbar_{0,n+2}\times_{\Lbar_n}\Lbar_n$, points of $\Mbar_n^r$ correspond to pairs $(\mu,\nu)\in\Mbar_{0,n+2}\times\Lbar_n$ such that $c(\mu)=q(\nu)$. The next lemma describes explicitly how every such point can be viewed as a heavy $(r,n)$-curve.

\begin{lemma}\label{lem:mod1}
Let $\mu=(C_\M;\{y_i\},y^\pm)\in\Mbar_{0,n+2}$ and $\nu=(C_\mcL;\{z_i\},z^\pm)\in\Lbar_n$ be such that $c(\mu)=q(\nu)$. Let $\tilde c:C_\mcM\rightarrow C_\mcL$ and $\tilde q:C_\mcL\rightarrow C_\mcL$ be the maps on curves induced by $c$ and $q$. Define $\zeta:C_\mcL\rightarrow C_\mcL$ to be multiplication by $\zeta$ on each component of $C_\mcL$. Then
\[
(C_\mcM\times_{C_\mcL}C_\mcL;\;\{(y_i,\zeta^jz_i)\},\;(y^\pm,z^\pm),\;\sigma)
\]
is a heavy $(r,n)$-curve, where $\sigma(y,z)=(y,\zeta z)$.
\end{lemma}

\begin{proof}
For convenience, set $C=C_\mcM\times_{C_\mcL} C_\mcL$.  Since $\tilde q:C_\mcL\rightarrow C_\mcL$ is a degree-$r$ cyclic cover of $C_\mcL$ that is fully ramified at the nodes and heavy marked points of $C_\mcL$, it follows that $C\rightarrow C_\mcM$ is a degree-$r$ cyclic cover of $C_\mcM$ that is fully ramified over $y^\pm$, as well as over all nodes of $C_\mcM$ that connect components in the chain from $y^-$ to $y^+$. Up to isomorphism, there is a unique such cyclic cover $C\rightarrow C_\mcM$. More specifically, $C$ is a nodal rational curve containing a central chain that we identify with $C_\mcL$, and for each tree $T\subseteq C_\mcM$ that is contracted by $\tilde c$, the curve $C$ contains $r$ identical such trees attached at simple nodes at the $r$ points of $\tilde q^{-1}(\tilde c(T))\subseteq C_\mcL$.

Given this description of $C$, one sees that it has an order-$r$ automorphism acting by multiplication by $\zeta$ on the central chain $C_\mcL$ and permuting the trees accordingly. The fiber product map $\tilde f_{\mcM}:C\rightarrow C_\mcM$ is the quotient by this automorphism, while the map $\tilde f_{\mcL}:C\rightarrow C_\mcL$ contracts the trees to the central chain, so this automorphism is the $\sigma$ in the statement of the lemma.

Set $x^{\pm}$ to be the endpoints of the central chain of $C$, corresponding to $z^\pm\in C_\mcL$.  Furthermore, for  each $i \in [n]$ and $j \in \Z_r$, set $z_i^j = \sigma^j(z_i^0)$, where $z_i^0$ is the unique element of the preimage $\tilde f_\mcM^{-1}(y_i)$ whose image under $\tilde f_\mcL$ is $z_i$. This data makes $C$ into a heavy $(r,n)$-curve, and because $(\tilde f_\mcM(z_i^j),\tilde f_\mcL(z_i^j))=(y_i,\zeta^jz_i)$ and $(\tilde f_\mcM(x^\pm),\tilde f_\mcL(x^\pm))=(y^\pm,z^\pm)$, it agrees with the data in the statement of the lemma.
\end{proof}

The previous lemma can be viewed as a map from $\Mbar^r_n$ to the set of heavy $(r,n)$-curves.  The next lemma, on the other hand, gives an explicit procedure for going in the other direction.

\begin{lemma}\label{lem:mod2}
Let $(C;\{z_i^j\},x^{\pm},\sigma)$ be a heavy $(r,n)$-curve.
\begin{itemize}
    \item Define $(C_\mcM;\{y_i\},y^{\pm})\in\Mbar_{0,n+2}$ by setting $C_{\mcM} = C/\sigma$, setting $y_i$ to be the image of the orbit $\{z_i^0,\dots,z_i^{r-1}\}$, and setting $y^\pm$ to be the image of $x^\pm$.
    \item Define $(C_\mcL;\{z_i\},z^\pm)\in \Lbar_n$ to be the image of $(C;\{z_i^j\},x^{\pm})\in\Mbar_{0,rn+2}$ under the forgetful and weight-reduction map $\Mbar_{0,rn+2}\rightarrow\Lbar_n$ that forgets all but the marked points $\{z_i^0\}$.
\end{itemize}
Then $C=C_\mcM\times_{C_\mcL}C_\mcL$, and under this identification of curves, $z_i^j=(y_i,\zeta^jz_i)$, $x^\pm=(y^\pm,z^\pm)$, and $\sigma(y,z)=(y,\zeta z)$. 
\end{lemma}

\begin{proof}
Notice that the canonical maps $C\rightarrow C_\mcM$ and $C\rightarrow C_\mcL$ give rise to a canonical map $C\rightarrow C_\mcM\times_{C_\mcL} C_\mcL$, by the universal property of fiber products. Given the explicit description of $C_\mcM\times_{C_\mcL} C_\mcL$ in the proof of Lemma~\ref{lem:mod1}, it is not hard to see that this map is an isomorphism of nodal curves, and that it identifies the marked points and automorphism as stated in the lemma.
\end{proof}

Because Lemma \ref{lem:mod1} describes a function from $\Mbar_n^r$ to the set of heavy $(r,n)$-curves, while Lemma~\ref{lem:mod2} describes its inverse, we conclude the following.

\begin{proposition}\label{prop:bijection}
    There is a bijection
    \[
    \Mbar_n^r=\{\textrm{\normalfont heavy }(r,n)\text{\normalfont -curves}\}.
    \]
\end{proposition}

Having now described how to view points of $\Mbar_n^r$ as heavy $(r,n)$-curves, we next describe how to carry out this procedure in families, thereby constructing a tautological family of heavy $(r,n)$-curves over $\Mbar_n^r$.

\subsection{Constructing a tautological family}

Recall that the universal curves over the moduli spaces $\Mbar_{0,n+2}$ and $\Lbar_n$ are given by the forgetful maps from the associated moduli spaces with one additional marked point: $\pi_\M:\Mbar_{0,n+3}\rightarrow\Mbar_{0,n+2}$ and $\pi_\mcL:\Lbar_{n+1}\rightarrow\Lbar_n$.  The map $\pi_\M$ has $n+2$ sections giving the marked points, which we denote $y_1,\dots,y_n,y^\pm$, while the map $\pi_\mcL$ has $n+2$ sections that we denote $z_1,\dots,z_n,z^\pm$. Consider the two corresponding fiber squares:
\[
\xymatrix{
   \Mbar_{n+1}^r \ar[r]^{\tilde{f}_{\mcL}}\ar[d]_{\tilde{f}_{\mcM}} & \Lbar_{n+1}\ar[d]^{\tilde{q}}\\
  \Mbar_{0,n+3} \ar[r]^-{\tilde{c}} & \Lbar_{n+1}}
  \hspace{2cm}
\xymatrix{
   \Mbar_n^r \ar[r]^{f_{\mcL}}\ar[d]_{f_\M} & \Lbar_n\ar[d]^{q}\\
  \Mbar_{0,n+2} \ar[r]^-{c} & \Lbar_n.}
\]
Since $q\circ\pi_\mcL\circ\tilde f_\mcL=c\circ\pi_\M\circ\tilde f_\M$, the universal property of fiber products yields a canonical map $\pi:\Mbar_{n+1}^r\rightarrow\Mbar_n^r$. We will argue that $\Mbar_{n+1}^r$ is a tautological family of heavy $(r,n)$-curves over $\Mbar_n^r$. In order to do so, we require an automorphism and $rn+2$ sections of this family, which we now define.

Let $\zeta:\Lbar_{n+1}\rightarrow\Lbar_{n+1}$ be the map that multiplies the final marked point on each chain by $\zeta$. We define an automorphism $\sigma:\Mbar_{n+1}^r\rightarrow\Mbar_{n+1}^r$ by applying the universal property of fiber products to the following diagram:
\begin{equation}
    \label{eq:sigmadiagram}
\xymatrix{
\Mbar_{n+1}^r \ar@{-->}[dr]^{\sigma}\ar@/^1.5pc/[drr]^{\zeta \circ \widetilde{f}_{\mcL}}\ar@/_1.5pc/[ddr]_{\widetilde{f}_{\M}} & & \\
  & \Mbar_{n+1}^r \ar[r]^{\widetilde{f}_{\mcL}}\ar[d]_{\widetilde{f}_{\M}} & \Lbar_{n+1}\ar[d]^{\widetilde{q}}\\
  & \Mbar_{0,n+3}\ar[r]^-{\widetilde{c}} & \Lbar_{n+1}.}
\end{equation}
Furthermore, we define sections $z_i^j$ of $\pi:\Mbar_{n+1}^r\rightarrow\Mbar_n^r$ by another application of the universal property of fiber products to the following diagram:
\[\xymatrix{
\Mbar_n^r\ar@{-->}[dr]^{z_i^j}\ar@/^1.5pc/[drr]^{\zeta^j\circ z_i\circ f_\mcL}\ar@/_1.5pc/[ddr]_{y_i\circ f_\M} & & \\
  & \Mbar_{n+1}^r \ar[r]^{\widetilde{f}_{\mcL}}\ar[d]_{\widetilde{f}_{\M}} & \Lbar_{n+1}\ar[d]^{\widetilde{q}}\\
  & \Mbar_{0,n+3}\ar[r]^-{\widetilde{c}} & \Lbar_{n+1},}\]
and we define the sections $x^{\pm}$ similarly.

With this notation established, we have the following result.

\begin{proposition}\label{prop:taut}
    The map $\pi:\Mbar_{n+1}^r\rightarrow\Mbar_n^r$ along with the automorphism $\sigma:\Mbar_{n+1}^r\rightarrow\Mbar_{n+1}^r$ and the $rn+2$ sections $z_i^j,x^+,x^-$ is a family of heavy $(r,n)$-curves over $\Mbar_n^r$. Moreover, this family is tautological in the sense that, for any $b\in\Mbar_n^r$, the tuple 
    \[
    \left(\pi^{-1}(b); \set{z_i^j(b)}, \; x^{\pm}(b), \; \sigma\vert_{\pi^{-1}(b)}\right)
    \]
    is isomorphic to the heavy $(r,n)$-curve parametrized by $b$ (according to Proposition~\ref{prop:bijection}).
\end{proposition}

\begin{proof}
We first confirm that $\pi$ is flat, which can be proven by covering $\Mbar_{n+1}^r$ by open sets on which $\pi$ is flat. Consider the following fiber square over $\Mbar_n^r$:
\[
\xymatrix{
   \Mbar_{n+1}^r \ar[r]^{\tilde{f}_{\mcL}}\ar[d]_{\tilde{f}_{\mcM}} & f_\mcL^*\Lbar_{n+1}\ar[d]^{\tilde{q}}\\
  f_\mcM^*\Mbar_{0,n+3} \ar[r]^-{\tilde{c}} & g^*\Lbar_{n+1},}
\]
where $g=c\circ f_\mcM=q\circ f_\mcL$. Since flatness is preserved by pullback, the pullbacks in this diagram are all flat over $\Mbar_n^r$. Notice that $\tilde{q}$ is \'etale away from the divisor $D \subseteq f_\mcL^*\Lbar_{n+1}$ of fixed points of $\zeta$. Since \'etaleness is preserved under base change, it follows that $\tilde{f}_{\mcM}$ is \'etale on $\Mbar_{n+1}^r\setminus\tilde f_\mcL^{-1}(D)$.  In particular, the restriction of $\pi$ to $\Mbar_{n+1}^r\setminus\tilde f_\mcL^{-1}(D)$ is the composition of an \'etale morphism and a flat morphism, so it is flat. On the other hand, the map $\tilde{c}$ is an isomorphism away from the divisor $D' \subseteq f_\mcM^*\Mbar_{0,n+3}$ of contracted trees, so it follows that $\tilde{f}_{\mcL}$ is an isomorphism on $\Mbar_{n+1}^r \setminus \tilde{f}_{\mcM}^{-1}(D')$.  Thus, the restriction of $\pi$ to this open locus is the composition of an isomorphism with a flat morphism, so it is flat.  Because $\widetilde{f}_{\mcL}^{-1}(D) \cap \widetilde{f}_{\mcM}^{-1}(D') = \emptyset$, we have covered $\Mbar_{n+1}^r$ by open sets on which $\pi$ is flat.

We next confirm that $\pi \circ \sigma = \pi$.  To do so, first note that
\[f_{\mcL} \circ \pi \circ \sigma = \pi_{\mcL} \circ \zeta \circ \widetilde{f}_{\mcL} = \pi_{\mcL} \circ \widetilde{f}_{\mcL},\]
where the first equality follows from \eqref{eq:sigmadiagram} and the second from the definition of $\zeta$.  This implies that the diagram
\[\xymatrix{
\Mbar_{n+1}^r \ar@{-->}[dr]\ar@/^1.5pc/[drr]^{\pi_{\mcL} \circ \widetilde{f}_{\mcL}}\ar@/_1.5pc/[ddr]_{\pi_{\mcM} \circ \widetilde{f}_{\mcM}} & & \\
  & \Mbar_{n}^r \ar[r]^{f_{\mcL}}\ar[d]_{f_{\M}} & \Lbar_{n}\ar[d]^{q}\\
  & \Mbar_{0,n+2}\ar[r]^-{c} & \Lbar_{n}.}\]
commutes with the dashed arrow equal either to $\pi \circ \sigma$ or to $\pi$, so the universal property of fiber products ensures that $\pi \circ \sigma = \pi$.

Finally, note that
\[
    \left(\pi^{-1}(b); \set{z_i^j(b)}, \; x^{\pm}(b), \; \sigma\vert_{\pi^{-1}(b)}\right)
\]
is isomorphic to the heavy $(r,n)$-curve parametrized by $b$; this follows from the observation that, upon restricting to the fiber over $b\in\Mbar_n^r$, the construction of $\Mbar_{n+1}^r$, $\sigma$, and $z_i^j,x^+,x^-$ over $\Mbar_n^r$ specializes to the construction of Lemma~\ref{lem:mod1} over $b$.  Thus, the given data indeed defines a tautological family of heavy $(r,n)$-curves.
\end{proof}

\subsection{The tautological family is universal}

We are now prepared to conclude with the following culminating result.

\begin{theorem}
\label{thm:modulispace}
The fiber product $\Mbar^r_n$ is a fine moduli space for the moduli problem of heavy $(r,n)$-curves.  That is, for any base scheme $B$, pulling back the tautological family gives a bijection
\[
\Bigg\{ \text{\normalfont morphisms }B \rightarrow \Mbar^r_n\Bigg\} \;\; \longrightarrow \;\; \left\{\substack{\textstyle \text{\normalfont families of heavy}\\ \textstyle (r,n)\text{\normalfont -curves over }B}\right\}.
\]
\end{theorem}

\begin{proof}
To prove injectivity, suppose that two morphisms $f,g:B\rightarrow\Mbar_n^r$ give rise to isomorphic families of heavy $(r,n)$-curves upon pulling back the tautological family. Then these isomorphic families restrict to isomorphic heavy $(r,n)$-curves over every $b\in B$. Proposition~\ref{prop:bijection} then implies that the $f(b)=g(b)$ for every $b\in B$, so $f=g$.

To prove surjectivity, let $\mcF=(\mcC;\{z_i^j\},x^{\pm},\sigma)$ be a family of heavy $(r,n)$-curves over $B$. Upon taking the quotient of $\mcC$ by $\sigma$, we obtain a family $\mcF_\mcM=(\mcC_\mcM;y_i,y^\pm)$ of curves in $\Mbar_{0,n+2}$, and thus a morphism $f_\mcM:B\rightarrow \Mbar_{0,n+2}$. Similarly, after contracting to the central chain of $\mcC$, we obtain a family $\mcF_\mcL=(\mcC_\mcL;z_i,z^\pm)$ of curves in $\Lbar_{n}$, and thus a morphism $f_\mcL:B\rightarrow\Lbar_n$. By the universal property of fiber products, $f_\mcM$ and $f_\mcL$ give rise to a morphism $f_\mcF:B\rightarrow\Mbar_n^r$. We claim that $\mcF\cong f_\mcF^*(\mcF^\mathrm{taut})$ where $\mcF^\mathrm{taut}$ denotes tautological family over $\Mbar_n^r$.

By construction of $\mcF^\mathrm{taut}$ as a fiber product, we can construct its pullback $f_\mcF^*(\mcF^\mathrm{taut})$ from the families $\mcF_\mcM$ and $\mcF_\mcL$:
\[
f_\mcF^*(\mcF^\mathrm{taut})=
(\mcC_\mcM\times_{\mcC_\mcL}\mcC_\mcL;\;\{(y_i,\zeta^jz_i)_i^j\},\;(y^\pm,z^\pm),\;\sigma).
\]
The natural maps $\mcC\rightarrow\mcC_\mcM$ and $\mcC\rightarrow\mcC_\mcL$ give rise to a map $\mcC\rightarrow \mcC_\mcM\times_{\mcC_\mcL}\mcC_\mcL$, and Lemma~\ref{lem:mod2} implies that this map is a fiberwise isomorphism of heavy $(r,n)$-curves over $B$. Since $\mcC$ is flat over $B$, the isomorphism on fibers implies that $\mcC\rightarrow \mcC_\mcM\times_{\mcC_\mcL}\mcC_\mcL$ is an isomorphism (see, for example, \cite[Section 37.16]{stacks-project}), and we then conclude that $\mcF\cong f_\mcF^*(\mcF^\mathrm{taut})$.
\end{proof}

Because Theorem~\ref{thm:modulispace} shows that $\Mbar^r_n$ is a fine moduli space for heavy $(r,n)$-curves and Proposition~\ref{prop:WCfiberproduct} shows that $\Mbar^r_n$ is the wonderful compactification of the $r$-braid arrangement with respect to its minimal building set, the proof of Theorem~\ref{thm:main} is complete.

\bibliographystyle{amsalpha}
\bibliography{references}
  
\end{document}